\documentclass[12pt]{article}
\usepackage[latin1]{inputenc}

\usepackage{amsfonts}
\usepackage{amssymb}
\usepackage{amsthm}
\usepackage{amsmath}
\usepackage{graphicx}
\usepackage{color}
\usepackage{amscd}
\usepackage{multirow}

\newcommand{\CC}{{\mathbb C}}
\newcommand{\C}{{\mathbb C}}
\newcommand{\cD}{{\mathcal D}}
\newcommand{\cG}{{\mathcal G}}
\newcommand{\cH}{{\mathcal H}}
\newcommand{\cO}{{\mathcal O}}
\newcommand{\cL}{{\mathcal L}}
\newcommand{\cQ}{{\mathcal Q}}
\newcommand{\cM}{{\mathcal M}}
\newcommand{\cF}{{\mathcal F}}
\newcommand{\cE}{{\mathcal E}}

\newcommand{\RR}{{\mathbb R}}
\newcommand{\R}{{\mathbb R}}
\newcommand{\ZZ}{{\mathbb Z}}
\newcommand{\Z}{{\mathbb Z}}
\newcommand{\NN}{{\mathbb N}}
\newcommand{\N}{{\mathbb N}}
\newcommand{\QQ}{{\mathbb Q}}

\newcommand{\PP}{{\mathbb P}}

\newcommand\inw{{{\operatorname{in}}_{\omega}}}

\newcommand\HAb {{H_A (\beta )}}
\newcommand\HHAb {{\mathcal{M}_A (\beta )}}

\newcommand\coefn{{(v)_{u_{-}}}}
\newcommand\coefd{{(v+u)_{u_{+}}}}
\newcommand\coefmu{{\frac{\coefn}{\coefd}}}

\newcommand\inww{{{\operatorname{in}}_{(-\omega , \omega )}}}

\newcommand\finw{{{\operatorname{fin}}_{\omega}(\HAb )}}
\newcommand\Irr{\operatorname{Irr}}
\newcommand\Sol{\operatorname{Sol}}
\newcommand\Per{\operatorname{Per}}

\newcommand{\nsup}{\operatorname{nsupp}}

\newtheorem{theorem}{Theorem}[section]
\newtheorem{proposition}[theorem]{Proposition}
\newtheorem{definition}[theorem]{Definition}

\newtheorem{lemma}[theorem]{Lemma}
\newtheorem{corollary}[theorem]{Corollary}
\newtheorem{remark}[theorem]{Remark}

\newtheorem{notation}[theorem]{Notation}

\title{Gevrey solutions of the irregular hypergeometric system associated with an affine monomial curve}
\author{M.C. Fern\'{a}ndez-Fern\'{a}ndez and F.J. Castro-Jim\'{e}nez \thanks{Both
authors partially supported by MTM2007-64509 and FQM333. The first
author is also supported by the FPU Grant AP2005-2360, MEC (Spain).
e.mail addresses: {\tt mcferfer@us.es}, {\tt castro@us.es}}\\
Departamento de \'{A}lgebra \\ Universidad de Sevilla}
\date{20 November 2008}

\begin{document}

\maketitle
\begin{abstract}
We describe the Gevrey series solutions at singular points of the
irregular hypergeometric system (GKZ system) associated with an
affine monomial curve. We also describe the  irregularity complex of
such a  system with respect to its singular support.
\end{abstract}

\section*{Introduction}
We explicitly describe the Gevrey solutions of the hypergeometric
system associated with an affine  monomial  curve in $\CC^n$ by
using $\Gamma$--series introduced by I.M. Gel'fand, M.I. Graev, M.M.
Kapranov and A.V. Zelevinsky  \cite{GGZ}, \cite{GZK} and also used
by M. Saito, B. Sturmfels and N. Takayama \cite{SST} in a very
useful and slightly different form. We use these $\Gamma$--series to
study the {\em Gevrey filtration of the irregularity complex} --as
defined by Z. Mebkhout in \cite[Sec. 6]{Mebkhout}-- of the
corresponding analytic hypergeometric $\cD$--module with respect to
its singular support. Here $\cD$ is the sheaf of linear differential
operators with holomorphic coefficients on $\CC^n$.

A general hypergeometric system $M_A(\beta)$ is associated with a
pair $(A,\beta)$ where $A=(a_{ij})$ is an integer $d\times n$ matrix
of full rank $d$ and $\beta$ is a parameter vector in $\CC^d$
(\cite{GGZ}, \cite{GZK1}, \cite{GZK}). The system $M_A(\beta)$ is
defined by the following system of linear partial differential
equations in the unknown $\varphi$:
$$ \prod_{i=1}^n\left(\frac{\partial  }{\partial x_i}\right)^{u_i}(\varphi)-
\prod_{i=1}^n\left(\frac{\partial }{\partial
x_i}\right)^{v_i}(\varphi )=0 \, \, {\mbox{ for }} (u,v\in \NN^n, \,
Au=Av)
$$
$$\sum_{j=1}^n a_{ij}x_j \frac{\partial \varphi }{\partial x_j}
-\beta_i\varphi =0 \, \, {\mbox{ for }} i=1,\ldots,d.
$$

If $d=1$ (i.e. if $A=(a_1\, a_2 \, \cdots \, a_n)$) we say that the
hypergeometric system $M_A(\beta)$ is associated with the affine
monomial curve defined by $A$ in $\CC^n$. If $d=2$ and the vector
$(1,\cdots,1)$ is in the $\QQ$-row span of $A$ then we say that
$M_A(\beta)$ is associated with the projective monomial curve
defined by $A$ in $\PP_{n-1}(\CC)$.

Rational solutions of the hypergeometric system associated with a
projective monomial curve have  been studied in
\cite{cattani-dandrea-dickenstein-99}. The holomorphic solutions of
general hypergeometric systems at a generic point in $\CC^n$  have
been widely studied (see e.g. \cite{GZK1}, \cite{GZK},
\cite{Adolphson}, \cite{SST}, \cite{ohara-takayama-2007}). The rank
of a general $M_A(\beta)$ is by definition the number of linearly
independent holomorphic solutions of $M_A(\beta)$ at a generic point
in $\CC^n$. This rank  equals, for generic $\beta\in \CC^d$ and
assuming $\ZZ A= \ZZ^d$, the normalized volume $vol(\Delta)$ of the
convex hull $\Delta$ of the columns of $A$ and the origin,
considered as points in $\RR^d$ (\cite{GZK} and \cite{Adolphson}).

Several recent papers are devoted to the study of the exceptional
parameters $\beta$, i.e. for which the rank jumps (see \cite{MMW05}
and the references therein, see also \cite{Berkesch}).

To setup the problem of computing the Gevrey solutions of a left
$\cD$--module, we will first recall the situation in the one
dimensional case.

Let $P=a_m \frac{d^m}{dx^m} + \cdots + a_1 \frac{d}{dx}+a_0$ be an
ordinary linear differential operator, of order $m$,  with
$a_i=a_i(x)$ a holomorphic function at the origin in $\CC$. The {\em
slopes} of $P$ are the slopes of the Newton polygon $N(P)$ of $P$
defined as the convex hull of the union of the quadrants
$(i,i-v(a_i)) + (\ZZ_{\leq 0})^2$ for $i=0,\ldots,m$, where $v(a_i)$
is the multiplicity of the zero of $a_i(x)$ at $x=0$. By Fuchs'
Theorem, $P$ is regular at $x=0$ if and only if $N(P)$ is the
quadrant $(m,m-v(a_m)) + (\ZZ_{\leq 0})^2$.

Malgrange-Komatsu's comparison Theorem  states that $P$ is regular
at $x=0$ if and only if the solution space
$\Sol(P;\widehat{\cO}/\cO)$ is zero (\cite{komatsu-index-1971},
\cite{Malgrange}). Here $P$ acts naturally  on the quotient
$\widehat{\cO}/\cO$ where $\widehat{\cO}=\CC[[x]]$ (resp.
$\cO=\CC\{x\}$) is the ring of formal (resp. convergent) power
series in the variable $x$. This solution space measures the
irregularity of $P$ at $x=0$ and will be denoted $\Irr_0(P):=\Sol(P;
\widehat{\cO}/\cO)$.

Moreover, the vector space $\Irr_0(P)$ is filtered by the so-called
Gevrey solutions of the equation $P(u)=0$. Let us denote by
$\widehat{\cO}_s \subset \widehat{\cO}$ the subring of Gevrey series
of order less than or equal to $s$ (where $s\geq 1$ is a real
number). A formal power series $f=\sum_i f_i x^i\in \widehat{\cO}$
is in $\widehat{\cO}_s$ if and only if the series $\rho_s(f):=\sum_i
(i!)^{1-s} {f_i} x^i$ is convergent at $x=0$. The irregularity
$\Irr_0(P)$ is filtered by the solution subspaces
$\Irr_0^{(s)}(P):=\Sol(P; \widehat{\cO}_s/\cO)$. By the comparison
Theorem of J.P. Ramis \cite{Ramis} the jumps in this filtration are
in  bijective correspondence with the slopes of $N(P)$ and the
dimension of each $\Irr_0^{(s)}(P)$ can also be read on $N(P)$.

In higher dimension, one has analogous results but the situation is
much more involved. Let $X$ be a complex manifold of dimension
$n\geq 1$ and let us denote by $\cO_X$ (resp. $\cD_X$) the sheaf of
rings of holomorphic functions on $X$ (resp. of linear differential
operators with holomorphic coefficients). The role of the previous
equation $P(u)=0$ is played here by a maximally overdetermined
system of linear partial differential equations on $X$, or
intrinsically by a holonomic $\cD_X$--module $\cM$. Z. Mebkhout
introduced in \cite{Mebkhout} the irregularity of $\cM$ along a
smooth hypersurface $Y\subset X$ as the complex of sheaves of vector
spaces $\Irr_Y(\cM):=\RR \cH om_{\cD_X}(\cM, \cQ_Y)$ where $\cQ_Y$
is the quotient $\cQ_Y=\frac{\cO_{\widehat{X|Y}}}{\cO_{X|Y}}$ and
$\cO_{\widehat{X|Y}}$ is the formal completion of $\cO_X$ with
respect to $Y$. This irregularity is nothing but the solution
complex of $\cM$ with values in $\cQ_Y$. Z. Mebkhout also introduced
the irregularity of order $s$ of $\cM$ along $Y$ as the solution
complex of $\cM$ with values in $\cQ_Y(s)$, i.e.
$\Irr_Y^{(s)}(\cM):=\RR \cH om_{\cD_X}(\cM, \cQ_Y(s))$ where
$\cQ_Y(s)$ is the quotient of the Gevrey series
${\cO_{\widehat{X|Y}}(s)}$  of order less than or equal to $s$ along
$Y$ by the sheaf $\cO_{X\vert Y}$. Z. Mebkhout proved
\cite{Mebkhout} that these irregularity complexes belong to
$\Per(\CC_Y)$, the abelian category of perverse sheaves on $Y$ and
that $\Irr_Y^{(s)}(\cM)$ is an increasing filtration of the
irregularity $\Irr_Y(\cM)$. A real number $s>1$ is said to be an
{\em analytic slope}  of $\cM$ along $Y$ at a point $p\in Y$ if the
inclusion $\Irr_Y^{(s')}(\cM)\subset \Irr_Y^{(s)}(\cM)$ is strict in
a neighborhood of $p\in Y$ for $1<s'<s$.

On the other hand, Y. Laurent defined the {\em algebraic slopes} of
$\cM$ by using  interpolations of the order filtration and the
Malgrange-Kashiwara filtration along $Y$ \cite{Laurent-ens-87}. Y.
Laurent also proved that these algebraic slopes form a finite set of
rational numbers.

The comparison Theorem of Y. Laurent and Z. Mebkhout
\cite{Laurent-Mebkhout} proves that the analytic and the algebraic
slopes coincide. Although the slopes of a $\cD_X$--module $\cM$ can
be computed by an algorithm when $\cM$ is algebraic \cite{ACG},
computations of the irregularity $\Irr_Y^{(s)}(\cM)$ are very rare.
More precisely, under some conditions on the matrix $A$, the slopes
of a hypergeometric system $M_A(\beta)$ can be combinatorially
described from the matrix $A$ (\cite{schulze-walther}, see also
\cite{Castro-Takayama}, \cite{hartillo_trans}).

In this paper we explicitly describe the irregularity
$\Irr_Y^{(s)}(\cM)$ for any $s$ and for $\cM=\cM_A(\beta)$ the
hypergeometric system associated with an affine monomial curve when
the associated semigroup is positive. Here $Y$ is the singular
support of the system.

The paper has the following structure. In Sections
\ref{Gevrey-series} and \ref{subsect-irregularity-slopes} we recall
the definition of Gevrey series and some essential results, proved
by Z. Mebkhout \cite{Mebkhout}, about the {\em irregularity complex}
$\Irr_Y(\cM)$ of a holonomic $\cD$-module $\cM$ with respect to a
hypersurface $Y$ in a complex manifold. We also recall the
comparison theorem of Y. Laurent and Z. Mebkhout for the {\em
algebraic } and {\em analytic slopes } of a holonomic $\cD$--module
\cite{Laurent-Mebkhout}.

Section \ref{GGZ-GKZ-systems} deals with the basic properties of the
$\Gamma$--series associated with the pair $(A,\beta)$ following
\cite{GGZ}, \cite[Sec. 1]{GZK} and \cite[Sec. 3.4]{SST}.

Sections \ref{case-smooth-monomial-curve} and
\ref{case_monomial_curve} are  devoted to the computation of the
cohomology of the irregularity complex $\Irr_Y^{(s)}(\cM_A(\beta))$
of the hypergeometric system $\cM_A(\beta)$ associated with a smooth
monomial curve and with a general monomial curve respectively, at
any point of the singular support $Y$ of the system. In both cases
the associated semigroup is assumed to be  positive.

Proofs are fulfilled by reducing the number of variables, by using
the restriction functor in $\cD$--module theory, and then applying
the results from the 2 dimensional case
\cite{fernandez-castro-dim2-2008}.

This paper can be considered as a natural continuation of
\cite{fernandez-castro-dim2-2008}, \cite{Castro-Takayama} and
\cite{hartillo_trans} and its results should be related to the ones
of \cite{takayama-modified-2007}. We have often used some essential
results of the book \cite{SST} about solutions of hypergeometric
systems. Some of our results are related to
\cite{oaku-on-regular-2007} and also to \cite{majima} and
\cite{iwasaki}.

The authors would  like to thank Ch. Berkesch,  N. Takayama  and
J.M. Tornero for their useful comments.

\section{Gevrey series.}\label{Gevrey-series}

Let $X$ be a complex manifold of dimension $n\geq 1$, ${\mathcal
O}_X $ (or simply $\cO$) the sheaf of  holomorphic functions on $X$
and ${\mathcal D}_X $ (or simply $\cD$) the sheaf of linear
differential operators with coefficients in $\cO_X$. The sheaf
$\cO_X$ has a natural structure of left $\cD_X$--module. Let  $Z$ be
a hypersurface in $X$ with defining ideal $\mathcal{I}_{Z}$. We
denote by $\cO_{X|Z}$ the restriction to $Z$ of the sheaf $\cO_X$
(and  we will also denote by $\cO_{X|Z}$ its extension by 0 on $X$).
Recall that the formal completion of $\cO_X$ along $Z$ is defined as

$$\cO_{\widehat{X|Z}}:= \lim_{\stackrel{\longleftarrow }{k}} \cO_{X} /
\mathcal{I}_{Z}^k.
$$

By definition $\cO_{\widehat{X|Z}}$ is a sheaf on  $X$ supported on
$Z$ and has a natural structure of left $\cD_X$--module. We will
also denote by $\cO_{\widehat{X|Z}}$ the corresponding sheaf on $Z$.
We denote by $\cQ_Z$ the quotient sheaf defined by the following
exact sequence
$$0\rightarrow \cO_{X|Z} \longrightarrow \cO_{\widehat{X|Z}}
\longrightarrow \cQ_Z \rightarrow 0.$$ The sheaf $\cQ_Z$ has then a
natural structure of left $\cD_X$--module.

\begin{remark} If $X=\CC$ and $Z=\{0\}$ then $\cO_{\widehat{X|Z},0}$ is nothing but
$\CC[[x]]$ the ring of formal power series in one variable $x$,
while $\cO_{\widehat{X|Z},p}=0$ for any nonzero $p\in X$. In this
case $\cQ_{Z,0} = \frac{\CC[[x]]}{\CC\{x\}}$ and $\cQ_{Z,p} = 0$ for
$p\not=0$.
\end{remark}

\begin{definition} Assume $Y\subset X$ is a smooth  hypersurface and that
around a point $p\in X$ the hypersurface $Y$ is locally defined by
$x_n=0$ for some system of local coordinates $(x_1,\ldots,x_n)$. Let
us consider a real number $s\geq 1$. A germ $f=\sum_{i\geq 0} f_i
(x_1 ,\ldots ,x_{n-1} ) x_n^i \in \cO_{\widehat{X|Y},p} $ is said to
be a Gevrey series of order $s$ (along $Y$ at the point $p$) if the
power series
$$\rho_s (f) := \sum_{i\geq 0} \frac{1}{i!^{s-1}} f_i (x_1
,\ldots ,x_{n-1} ) x_n^i $$ is convergent  at $p$.
\end{definition}

The sheaf $\cO_{\widehat{X|Y}}$ admits a natural filtration by the
sub-sheaves $\cO_{\widehat{X|Y}}(s)$ of Gevrey series of order $s$,
$1\leq s\leq  \infty$ where by definition
$\cO_{\widehat{X|Y}}(\infty)= \cO_{\widehat{X|Y}}$. So we have
$\cO_{\widehat{X|Y}}(1) = \cO_{{X|Y}}$. We can also consider the
induced filtration on $\cQ_Y$, i.e. the filtration by the
sub-sheaves $\cQ_{Y} (s)$ defined by the exact sequence:
\begin{align}\label{sucs}
0\rightarrow \cO_{X|Y} \longrightarrow \cO_{\widehat{X|Y}} (s)
\longrightarrow \cQ_Y (s) \rightarrow 0
\end{align}

\begin{definition}\label{def-gevrey-index}
Let $Y$ be a smooth  hypersurface in  $X=\CC^n$ and let $p$ be a
point in $Y$. The Gevrey index of a formal power series $f\in
{\cO_{\widehat{X|Y},p}}$ with respect to $Y$ is the smallest $ 1
\leq s\leq \infty$ such that $f \in {\cO_{\widehat{X|Y}}}(s)_p$.
\end{definition}

\section{Irregularity complex and
slopes}\label{subsect-irregularity-slopes}
Let $X$ be a complex manifold. We recall here the definition of the
irregularity (also called the irregularity complex) of a left
coherent $\cD_X$--module given by Z. Mebkhout \cite[(2.1.2) and page
98]{Mebkhout}.

Recall that if $\cM$ is a coherent left $\cD_X$--module and $\cF$ is
any $\cD_X$--module, the {\em solution complex} of $\cM$ with values
in $\cF$ is by definition the complex $$\RR \cH
om_{\cD_X}(\cM,\cF)$$ which is an object of $D^b(\CC_X)$ the derived
category of bounded complexes of sheaves of $\CC$--vector spaces on
$X$. The cohomology sheaves of the solution complex are then $\cE
xt^i_{\cD_X}(\cM,\cF)$ (or simply $\cE xt^i(\cM,\cF)$) for $i\in
\NN$.

\begin{definition} {\rm \cite[(2.1.2) and page
98]{Mebkhout}} Let $Z$ be a hypersurface in $X$. The irregularity of
$\mathcal{M}$ along $Z$ (denoted by $\operatorname{Irr}_Z(\cM)$) is
the solution complex of $\cM$ with values in $\cQ_Z$, i.e.
$$\operatorname{Irr}_Z (\mathcal{M}):=\RR \cH om_{\cD_{X}}
(\mathcal{M},\cQ_Z )$$
\end{definition}

If $Y$ is a smooth hypersurface in $X$  we  also have  the following
definition (see \cite[D\'{e}f. 6.3.7]{Mebkhout})

\begin{definition}
For each $1\leq s \leq \infty$, the irregularity of order $s$ of
$\cM$ with respect to $Y$ is the complex
$\operatorname{Irr}_{Y}^{(s)} (\cM ):=\R \cH om_{\cD_X }(\cM , \cQ_Y
(s))$.
\end{definition}

\begin{remark} Since $\cO_{\widehat{X|Y}}(\infty)= \cO_{\widehat{X|Y}}$
we have $\operatorname{Irr}_{Y}^{(\infty)} (\cM ) =
\operatorname{Irr}_{Y}^{} (\cM )$.  The support of the irregularity
of $\cM$ along $Z$ (resp. $\operatorname{Irr}_{Y}^{(s)} (\cM )$) is
contained in $Z$ (resp. in $Y$).

If $X=\CC$, $Z=\{0\}$ and $\cM=\cD_X/\cD_X P$ is the $\cD_X$--module
defined by some nonzero linear differential operator
$P(x,\frac{d}{dx})$ with holomorphic coefficients, then
$\Irr_Z(\cM)$ is represented by the complex
$$ 0 \longrightarrow \frac{\CC[[x]]}{\CC\{x\}}
\stackrel{P}{\longrightarrow} \frac{\CC[[x]]}{\CC\{x\}}
\longrightarrow 0$$ where $P$ acts naturally on the quotient
$\frac{\CC[[x]]}{\CC\{x\}}$.
\end{remark}

Z. Mebkhout have proved \cite[Th. 6.3.3]{Mebkhout} that for any
holonomic $\cD_X$--module $\cM$ the complex $\Irr_Y^{(s)}(\cM)$ is a
perverse sheaf on the smooth hypersurface $Y\subset X$ for any
$1\leq s\leq \infty$.


A complex $\cF^\bullet \in D^b(\CC_X)$ of sheaves of vector spaces
is said to be {\em constructible} if there exists a stratification
$(X_\lambda)$ of $X$ such that the cohomology sheaves
$\cH^i(\cF^\bullet)$ are local systems on each $X_\lambda$. A
constructible complex $\cF^\bullet$ satisfies the {\em support
condition} on $X$ if the following conditions hold:
\begin{enumerate} \item $\cH^i(\cF^\bullet)=0$ for $i<0$ and $i > n=dim (X)$.
\item The dimension of the support of $\cH ^i(\cF^\bullet)$ is less than
 or equal to $n-i$ for
$0\leq i\leq n.$
\end{enumerate}

A constructible complex  $\cF^\bullet$ is said to be {\em perverse}
on $X$ (or even a {\em perverse sheaf} on $X$) if both $\cF^\bullet$
and its dual $\RR \cH om_{\CC_X}(\cF^\bullet,\CC_X)$ satisfy the
support condition.

The category $\operatorname{Per}(\CC_X)$ of perverse sheaves on $X$
is an abelian category (see \cite{BBD}).

\begin{remark}
From \cite[Cor. 6.3.5]{Mebkhout} each
$\operatorname{Irr}_{Y}^{(s)}(-) $ for $1\leq s \leq \infty$, is an
exact functor from the category of holonomic $\cD_X$-modules to the
category of  perverse sheaves on $Y$.

Moreover, the sheaves $\operatorname{Irr}_{Y}^{(s)} (\cM )$, $1\leq
s \leq \infty$,  form an increasing filtration of
$\operatorname{Irr}_{Y}^{(\infty )} (\cM )= \operatorname{Irr}_{Y}
(\cM )$. This filtration is called the Gevrey filtration of
$\Irr_Y(\cM)$ (see \cite[Sec. 6]{Mebkhout}).
\end{remark}

Let us denote by
$$\operatorname{Gr}_s(\operatorname{Irr}_Y(\cM)):=
\frac{\Irr^{(s)}_Y(\cM)}{\Irr^{(<s)}_Y(\cM)}$$ for $1\leq s \leq
\infty$ the graded object associated with the Gevrey filtration of
the irregularity $\Irr_Y(\cM)$ (see \cite[Sec.
2.4]{Laurent-Mebkhout}).

We say, with \cite[Sec. 2.4]{Laurent-Mebkhout}, that $1\leq s <
\infty$ is an {\em analytic slope} of $\cM$ along $Y$ at a point
$p\in Y$ if $p$ belongs to the closure of the support of
$\operatorname{Gr}_s(\operatorname{Irr}_Y(\cM))$. Y. Laurent
(\cite{Laurent-ast-85}, \cite{Laurent-ens-87}) defines, in an
algebraic way,  the {\em algebraic slopes} of any coherent
$\cD_X$--module $\cM$ along $Y$. These algebraic slopes can be
algorithmically computed if the module $\cM$ is defined by
differential operators with polynomial coefficients \cite{ACG}.

In \cite[Th. 2.5.3]{Laurent-Mebkhout} Y. Laurent and Z. Mebkhout
prove that for any holonomic $\cD_X$--module  the analytic and the
algebraic slopes, with respect to any smooth hypersurface, coincide
and that they are rational numbers. In \cite{Castro-Takayama} and
\cite{hartillo_trans} are described the slopes at the origin (with
respect to any hyperplane in $\CC^n$) of the hypergeometric system
associated to any affine monomial curve. In \cite{schulze-walther}
M. Schulze and U. Walther describe the slopes of any hypergeometric
system with respect to any coordinate variety in $\CC^n$ under some
assumption on the semigroup associated with the system. By technical
reasons the definition of slope given in \cite{Castro-Takayama} and
\cite{hartillo_trans} is slightly different to the one of Y.
Laurent: a real number $-\infty \leq r \leq 0$ is called a slope in
\cite{Castro-Takayama} and \cite{hartillo_trans} if $\frac{r-1}{r}$
is an algebraic slope in the sense of Y. Laurent
\cite{Laurent-ens-87}.

\section{Hypergeometric systems and $\Gamma$--series}\label{GGZ-GKZ-systems}
Hypergeometric systems are defined on $X=\CC^n$. We denote by
$A_n(\CC)$ or simply $A_n$ the complex Weyl algebra of order $n$,
i.e. the ring of linear differential operators with coefficients in
the polynomial ring $\CC[x]:=\CC[x_1,\ldots,x_n]$. The partial
derivative $\frac{\partial}{\partial x_i}$ will be denoted by
$\partial_i$.

Let $A=(a_{ij})$ be an integer   $d\times n$ matrix with rank $d$
and $\beta\in \CC^d$. Let us denote by $E_i(\beta)$ for $i=1,\ldots,
d$, the operator $E_i(\beta) := \sum_{j=1} ^n a_{ij}x_j\partial_j
-\beta_i$. The toric ideal $I_A \subset
\CC[\partial]:=\CC[\partial_1,\ldots,\partial_n] $ associated with
$A$ is generated by the binomials
$\Box_u:=\partial^{u_+}-\partial^{u_{-}}$ for $u\in \ZZ^n$ such that
$Au=0$ where $u=u_+-u_-$ and  $u_+, u_-$ are both in $ \NN^n$ and
have  disjoint support.

The left ideal $A_n I_A + \sum_i A_n E_i(\beta) \subset A_n$ is
denoted by $H_A(\beta)$ and it will be called the {\em
hypergeometric ideal} associated with $(A,\beta)$. The (global)
hypergeometric module associated with $(A,\beta)$ is by definition
(see \cite{GGZ}, \cite{GZK}) the quotient
$M_A(\beta):=A_n/H_A(\beta)$.

When $X=\CC^n$ is considered as complex manifold, to the pair
$(A,\beta)$ we can also associate the corresponding analytic
hypergeometric $\cD_X$--module, denoted by $\cM_A(\beta)$, which is
the quotient of $\cD_X$ modulo the sheaf of left ideals in $\cD_X$
generated by $H_A(\beta)$.

%
In what follows we will use
$\Gamma$--series following \cite{GGZ} and \cite[Sec. 1]{GZK} and in
the way these objects are handled in \cite[Sec. 3.4]{SST}.

Let the pair $(A,\beta)$ be as before. Assume $v\in \CC^n$. We will
consider the $\Gamma$--series  $$ \varphi_{v} := x^v \sum_{u\in L_A}
\frac{1}{\Gamma(v+u+{\bf 1})}x ^u \in x^v\CC[[x_1^{\pm 1},\ldots,
x_n^{\pm n}]]
$$ where ${\bf 1} =(1,1,\ldots,1)\in \NN^n$, $L_A=\ker_\ZZ(A)$ and for $\gamma=(\gamma_1,\ldots,\gamma_n)\in \CC^n$
one has by definition $\Gamma(\gamma)=\prod_{i=1}^{n}
\Gamma(\gamma_i)$ (where $\Gamma$ is the Euler gamma function).
Notice that the set $x^v\CC[[x_1^{\pm 1},\ldots, x_n^{\pm n}]]$ has
a natural structure of left $A_n(\CC)$-module although it is not a
$\cD_{X,0}$--module. Nevertheless, if  $Av=\beta$ then the
expression $\varphi_v$ formally satisfies the operators defining
$\cM_A(\beta)$. Let us notice that if $u\in L_A$ then $\varphi_v =
\varphi_{v+u}$.

If $v\in (\CC\setminus \ZZ_{<0})^n$ then the coefficient
$\frac{1}{\Gamma(v+u+{\bf 1})}$ is non-zero for all $u\in L_A$ such
that $u_i+v_i\geq 0$ for all $i$ with $v_i\in \NN$. We also have the
following equality
\begin{align}\label{Gamma-factorial} \frac{\Gamma(v+{\bf
1})}{\Gamma(v+u+{\bf 1})} = \frac{(v)_{u_-}}{(v+u)_{u_+}}\end{align}
where for any $z\in \CC^n$ and any $\alpha \in \NN^n$ we have the
convention
$$(z)_\alpha = \prod_{i, \, \alpha_i>0} \prod_{j=0}^{\alpha_i-1} (z_i-j).$$

Following \cite[p. 132-133]{SST} the {\em negative support} of $v$
(denoted by $\nsup(v)$)  is the set of indices $i$ such that $v_i\in
\ZZ_{<0}$. We say that $v$ has {\em minimal negative support} if
there is no $u\in L_A$ such that $\nsup(v+u)$ is a proper subset of
$\nsup(v)$.

Assume  $v\not\in (\CC\setminus \ZZ_{<0})^n$ has minimal negative
support. The negative support of $v$ is then a non-empty set and
$\Gamma(v+{\bf 1})=\infty$. Moreover for each $u\in L_A$ at least
one coordinate of $v+u$ must be strictly negative (otherwise
$\nsup(v+u)=\emptyset \subsetneq  \nsup(v)$). So $\Gamma(v+u+{\bf
1})=\infty$ for all $u\in L_A$ and $\varphi_v=0$.

If $v\not\in (\CC\setminus \ZZ_{<0})^n$ does not have minimal
negative support then there exists ${u}\in L_A$ such that $v+{u}$
has minimal negative support. If $\nsup(v+u)=\emptyset$ then
$\varphi_v=\varphi_{v+u}\not= 0$ while if $\nsup(v+u)\not=\emptyset$
then $\varphi_v=\varphi_{v+u}= 0$.

Following {\it loc. cit.}, for any $v\in \CC^n$ we will consider the
series
$$\phi_v := x^v \sum_{u\in N_v} \coefmu x^u $$ where $N_v=\{u\in
L_A \, \vert \, \nsup(v+u)=\nsup(v)\}$.

If $Av=\beta$ then $\phi_v$ is a solution of the hypergeometric
ideal $H_A(\beta)$ (i.e. $\phi_v$ is formally annihilated by
$H_A(\beta)$)  if and only if $v$ has minimal negative support
\cite[Prop. 3.4.13]{SST}.

For $v\in (\CC\setminus \ZZ_{<0})^n$ we have $$\frac{\Gamma(v+{\bf
1})}{\Gamma(v+u+{\bf 1})}=\coefmu
$$ and $\Gamma(v+{\bf 1})\varphi_v=\phi_v$.

If $v\not\in (\CC\setminus \ZZ_{<0})^n$ then the coefficient of
$x^v$ in  $\phi_v$ is non-zero (in fact this coefficient is 1) while
it is zero in $\varphi_v$.

In order to simplify notations we will adopt in the sequel the
following convention: for $v\in \CC^n$ and $u\in L_A$ we will denote
$$\Gamma[v;u]:=\coefmu$$ if $u\in N_v$ and $\Gamma[v;u]:=0$
otherwise. With this convention we have $$ \phi_v = x^v \sum_{u\in
L_A} \Gamma [v;u] x^u.$$

In what follows  we will describe the irregularity, along the
singular locus, of the hypergeometric system associated with an
affine monomial curve in $X=\CC^n$ and with a parameter $\beta \in
\CC$.

\section{The case of a smooth monomial
curve}\label{case-smooth-monomial-curve}
Let $A=(1\, a_2 \, \cdots \, a_n)$ be an integer row matrix with $1
< a_2 < \cdots < a_n$ and $\beta \in \CC$. Let us denote by
$\cM_A(\beta)$ the corresponding analytic hypergeometric system on
$X=\CC^n$. We will simply denote $\cD$ the sheaf $\cD_X$ of linear
differential operators with holomorphic coefficients.

Although it can be deduced from general results (see \cite{GGZ} and
\cite[Th. 3.9]{Adolphson}), a direct computation shows in this case
that the characteristic variety of $\cM_A(\beta)$ equals $T^*_X X
\cup T^*_Y X$ where $Y=(x_n =0)$. The module $\cM_A(\beta)$ is then
holonomic and  its singular support is $Y$. Let us denote by
$Z\subset \CC^n$ the hyperplane $x_{n-1}=0$.

One of the  main results in this Section is
\begin{theorem}\label{teorext}
Let $A=(1 \; a_2 \; \cdots \; a_n )$ be an integer row matrix with
$1<a_2 <\cdots <a_n$ and $\beta\in \C$. Then the cohomology sheaves
of $\operatorname{Irr}^{(s)}_Y (\HHAb)$ satisfy:
\begin{enumerate}
\item[i)] $\mathcal{E}xt^0_{\cD}
(\HHAb , \cQ_Y (s))=0 $ for  $1\leq s < a_n / a_{n-1}$.
\item[ii)] $\mathcal{E}xt^0_{\cD} (\HHAb , \cQ_Y (s))_{|Y\cap Z}=0 $ for $1\leq s\leq \infty$.
\item[iii)] $\dim_{\CC}\left(\mathcal{E}xt^0_{\cD} (\HHAb ,\cQ_Y (s))_p \right)=a_{n-1}$ for all $s\geq a_n /a_{n-1}$ and
$p\in Y\setminus Z$. \item[iv)] $\mathcal{E}xt^i_{\cD} (\HHAb ,
\cQ_Y (s))=0 $ for $i \geq 1$ and $1\leq s\leq \infty$.
\end{enumerate}
\end{theorem}

The main ingredients  in the proof  of Theorem \ref{teorext} are:
Corollary \ref{sumadirectan}, the corresponding results for the case
of monomial plane curves \cite{fernandez-castro-dim2-2008},
Cauchy-Kovalevskaya Theorem for Gevrey series (see \cite[Cor.
2.2.4]{Laurent-Mebkhout2}), the perversity of
$\Irr_Y^{(s)}(\cM_A(\beta))$ \cite[Th. 6.3.3]{Mebkhout} and
Kashiwara's constructibility theorem \cite{kashiwara-overdet-75}.

We will also describe a basis of the solution vector space in part
{\em iii)} of Theorem \ref{teorext} (see Theorem
\ref{basis_of_ext_i_Q_s}).

\subsection{Reduction of the number of variables by restriction}

In the sequel we will use some  results concerning restriction of
hypergeometric systems.

\begin{theorem}{\rm  \cite[Th. 4.4]{Castro-Takayama}}\label{isocastro-takayama}
Let  $A=(1 \; a_2 \; \cdots \; a_n )$ be an integer row matrix with
$1<a_2 <\cdots <a_n$ and $\beta\in \C$. Then for $i=2,\ldots ,n$,
the restriction of $\HHAb$ to  $(x_i = 0)$ is isomorphic to the
$\cD'$-module
$$\cM_A(\beta)_{|(x_i=0)}:=\frac{\cD}{\cD \HAb + x_i \cD} \cong  \frac{\cD '}{\cD ' H_{A'}(\beta
)}$$ where  $A' =(1 \; a_2 \; \cdots \; a_{i-1} \; a_{i+1} \; \cdots
\; a_n )$ and $\cD '$ is the sheaf of linear differential operators
with holomorphic coefficients on $\CC^{n-1}$ (with coordinates $x_1
,\dots , x_{i-1} , x_{i+1} ,\dots ,x_n $).
\end{theorem}

\begin{theorem}\label{sumadirecta}
Let  $A=(1\; k a \; k b)$ be an integer row matrix with  $1\leq a <
b$, $1  < ka < kb$  and $a,b$ relatively prime. Then for all  $\beta
\in \CC$ there exist $\beta_0 ,\ldots ,\beta_{k-1}\in \CC$ such that
the restriction of $\HHAb$ to  $(x_1 = 0)$ is isomorphic to the
$\cD'$-module
$$\cM_A(\beta)_{|(x_1=0)}:=\frac{\cD}{\cD \HAb + x_1 \cD} \simeq \bigoplus_{i=0}^{k-1}
\mathcal{M}_{A'}(\beta_i )$$ where $\cD'$ is the sheaf of linear
differential operators on the plane $(x_1=0)$ and $A' = (a \; b)$.
Moreover, for all but finitely many $\beta\in \CC $ we can take
$\beta_i = \frac{\beta -i}{k} $, $i=0,1,\ldots ,k-1$.
\end{theorem}

An ingredient in the proof of Theorem \ref{teorext} is the following

\begin{corollary}\label{sumadirectan}
Let  $A=(1 \; a_2 \; \cdots \;  a_{n})$ be an integer row matrix
with  $1<a_2 < \cdots <a_n$ and $\beta \in \CC$. Then there exist
$\beta_i\in \CC$, $i=0,\ldots,k-1$ such that the restriction of
$\HHAb$ to $X'=( x_1=x_2=\cdots = x_{n-2} = 0)$ is isomorphic to the
$\cD'$-module
$$\cM_A(\beta)_{|X'}:=\frac{\cD}{\cD \HAb + ( x_1, x_2,\cdots ,x_{n-2} )\cD}
\simeq \bigoplus_{i=0}^{k-1} \mathcal{M}_{A'}(\beta_i )$$ where
$\cD'$ is the sheaf of linear differential operators on $X'$, $A'
= \frac{1}{k}(a_{n-1} \; a_n )$ and $k=\operatorname{gcd}(a_{n-1}
,a_n )$. Moreover, for all but finitely many $\beta\in \CC $ we
can take $\beta_i =\frac{\beta -i}{k}$, $i=0,1,\ldots ,k-1$.
\end{corollary}

Let us fix some notations and state some preliminaries results in
order to prove Theorem \ref{sumadirecta}.

\begin{notation}\label{notations_SST}  Let $A$ be an  integer $d\times n$--matrix of rank $d$
and $\beta \in \CC^n$. For any weight vector $\omega\in \RR^n$ and
any ideal  $J\subset
\CC[\partial]=\CC[\partial_1,\ldots,\partial_n]$ we denote by $\inw
(J)$ the initial ideal of $J$ with respect to the graduation on
$\CC[\partial]$ induced by $w$. According to \cite[p. 106]{SST} the
{\em fake initial ideal} of $H_A(\beta)$ is the ideal $\finw = A_n
\inw (I_A) + A_n (A\theta-\beta)$ where
$\theta=(\theta_1,\ldots,\theta_n)$ and $\theta_i = x_i
\partial_i$.

Assume now $A=(1\, \, ka\,\, kb)$ is an integer row matrix with
$1\leq a < b$, $1  < ka < kb$  and $a,b$ relatively prime.

Let us write $P_1 =
\partial_2^b -\partial_3^a$, $P_2 =\partial_1^{ka} -\partial_2$,
$P_3 =\partial_1^{kb}-\partial_3$ and  $E=\theta_1 + k a \theta_2 +
k b \theta_3 -\beta$. It is clear that  $P_1 \in \HAb =\langle P_2 ,
P_3, E\rangle \subset A_3$.

Let us consider  $\prec$ a monomial order on the monomials in $A_3$
satisfying:
$$\left. \begin{array}{l}
\gamma_1 + a \gamma_2 + b\gamma_3 < \gamma_1 ' + a \gamma_2 ' + b\gamma_3 '   \\
\mbox{{\rm or} } \\
\gamma_1 + a \gamma_2 + b\gamma_3 = \gamma_1 ' + a \gamma_2 ' +
b\gamma_3 ' \mbox{ {\rm and}  } 3 a \gamma_2 + 2 b\gamma_3 < 3 a
\gamma_2 '+ 2 b\gamma_3 '
\end{array}\right\}\Rightarrow x^{\alpha } \partial^{\gamma } \prec x^{\alpha '}
\partial^{\gamma '} $$

Write $\omega=(1,0,0)$  and let us denote by   $\prec_{\omega}$ the
monomial order on the monomials in $A_3$ defined  as
$$  x^{\alpha } \partial^{\gamma } \prec_{\omega} x^{\alpha '}
\partial^{\gamma '} \stackrel{\textmd{Def.}}{\Longleftrightarrow} \left\{ \begin{array}{l}
\gamma_1 -\alpha_1 < \gamma_1 '- \alpha_1 '   \\
\mbox{ {\rm or} }\\
\gamma_1 -\alpha_1 = \gamma_1 '- \alpha_1 ' \mbox{ {\rm and}   }
x^{\alpha }
\partial^{\gamma } \prec x^{\alpha '}
\partial^{\gamma '}
\end{array}\right.
$$
\end{notation}

\begin{lemma}\label{inwwHAb}
Let $A=(1\; k a \; k b)$ be an integer row matrix with $1\leq a <
b$, $1 < ka < kb$  and $a,b$ relatively prime. Then
$$\finw = A_3 \inw (I_A ) + A_3 E =A_3 (P_1 , E ,
\partial_1^k)$$  for  $\beta \notin
\N^{\ast}:=\NN\setminus\{0\}$ and for all  $\beta \in \N^{\ast}$ big
enough.
\end{lemma}

\begin{proof}
The proof  follows from the application of  Buchberger's algorithm
in the Weyl algebra $A_3$ and the fact that $\partial_1^k \in \finw$
for $\beta $  as in the statement.
\end{proof}

\begin{definition}\label{def-b-function}{\rm \cite[Def. 5.1.1]{SST}}
Let  $I\subseteq A_n $ be a holonomic ideal (i.e. $A_n/I$ is
holonomic) and $\widetilde{\omega}\in \R^n\setminus\{0\}$. The
$b$-function  $I$ with respect to  $\widetilde{\omega}$ is the monic
generator of the ideal
$$\operatorname{in}_{(-\widetilde{\omega} ,\widetilde{\omega})} (I)\cap \C[\tau ] $$ where
$ \tau=\widetilde{\omega}_1 \theta_1 +\cdots + \widetilde{\omega}_n
\theta_n$ and $\operatorname{in}_{(-\widetilde{\omega}
,\widetilde{\omega})} (I)$ is the  initial ideal of $I$ with respect
to the weight vector $(-\widetilde{\omega} ,\widetilde{\omega})$.
\end{definition}

\begin{corollary}\label{bfuncion}
Let $A=(1\; k a \; k b)$ be an integer row matrix with $1\leq a <
b$, $1 < ka < kb$  and $a,b$ relatively prime. Then the $b$-function
of $\HAb$ with respect to $\omega=(1,0,0)$ is $$b(\tau)=
\tau(\tau-1)\cdots (\tau-(k-1))$$ for all but finitely many
$\beta\in \CC$.
\end{corollary}


\begin{proof}
From \cite[Th. 3.1.3]{SST} 
for all but finitely many  $\beta \in \C$ we have
$$\inww (\HAb) =\finw .$$
Then by using Lemma \ref{inwwHAb} we get $$\inww (\HAb) =A_3(P_1 , E
,\partial_1^k)$$ for all but finitely many $\beta\in \C$.   An easy
computation shows that  $\{ P_1 , E ,\partial_1^k \}$ is a Groebner
basis of the ideal  $\inww (\HAb )$ with respect to any monomial
order $>$ satisfying  $\theta_3> \theta_1 ,\theta_2$ and
$\partial_2^b
>\partial_3^a$. In particular we can consider the lexicographic
order   $$x_3 > x_2 >
\partial_2 > \partial_3 > x_1 > \partial_1 $$
which is an elimination order for $x_1$ and $\partial_1$. So we get
$$\inww (\HAb) \cap \C [x_1]\langle \partial_1 \rangle = \langle
\partial_1^k \rangle$$ and since  $x_1^k \partial_1^k =\theta_1
(\theta_1 -1)\cdots (\theta_1 - (k-1))$, we have
$$\inww (\HAb) \cap \C [\theta_1 ]=\langle \theta_1 (\theta_1
-1)\cdots (\theta_1 - (k-1) )\rangle$$ This proves the corollary.
\end{proof}

\begin{remark}
Corollary \ref{bfuncion} can be related to \cite[Th.
4.3]{Castro-Takayama} which proves that for $A=(1 \; a_2 \; \cdots
\; a_n)$ with $1<a_2 <\cdots < a_n $, the $b$-function of  $\HAb$
with respect to $e_i$ is $b(\tau)=\tau$, for  $i=2,\ldots , n$. Here
$e_i \in \R^n$ is the vector with a 1 in the $i$-th coordinate and 0
elsewhere.
\end{remark}

Recall (see e.g. \cite[Def. 1.1.3]{SST}) that a Groebner basis of a
left ideal $I\subset A_n$ with respect to $(-\omega,\omega) \in
\RR^{2n}$ (or simply with respect to $\omega \in \RR^n$) is a finite
subset $G\subset I$ such that $I=A_n G$ and $\inww (I)=A_n \inww
(G)$ where $\inww (G)=\{\inww (P)\, |\, P\in G\}$.

\begin{lemma}\label{BGdeHAb}
Let  $A=(1\; k a \; k b)$ be an integer row matrix with  $1\leq a <
b$, $1<ka<kb$ and $a,b$ relatively prime. For all but finitely many
$\beta \in \C$, a Groebner basis of  $\HAb \subset A_3$ with respect
to $\omega=(1,0,0)$ is
$$\{ P_1 , P_2 , P_3 , E , R\} $$ for some  $R\in A_3$
satisfying  $\inww(R)=\partial_1^k$.
\end{lemma}

The following Proposition is a particular case of \cite[Th.
6.5.]{Berkesch} (see also \cite[Th. 2.1]{Saito01}) and it will be
used later.

\begin{proposition}\label{M-a-beta-isom}
Assume $A=(a_1 \; a_2 \cdots a_n)$ is an integer row matrix with $0
< a_1 < a_2 < \cdots < a_n$. For $\beta, \beta' \in \C^d$ we have
that $\HHAb \simeq \mathcal{M}_A (\beta ')$ if and only if one of
the following conditions holds:
\begin{enumerate}
\item[i)] $\beta , \beta ' \in \NN A$.
\item[ii)] $\beta , \beta ' \in \ZZ \setminus \NN A$.
\item[iii)] $\beta , \beta ' \notin \ZZ$ but $\beta - \beta
'\in \Z$.
\end{enumerate}
\end{proposition}

\begin{proof}(\textbf{Theorem  \ref{sumadirecta}})
We have $A=(1\; ka\; kb)$ with $1\leq a < b$, $1< ka < kb$ and $a,b$
relatively prime. From Proposition  \ref{M-a-beta-isom} it is enough
to compute the restriction for all but finitely many $\beta \in \C$.
We will compute the restriction of $\HHAb$ to $(x_1 = 0)$ by using
an algorithm by T. Oaku and N. Takayama \cite[Algorithm 5.2.8]{SST}.


Let $r=k-1$ be  the biggest integer root  of the Bernstein
polynomial $b(\tau)$ of $\HAb$ with respect to $\omega=(1,0,0)$ (see
Corollary \ref{bfuncion}). Recall that $\cD'=\cD_{\CC^2}$ and that
we are using $(x_2,x_3)$ as coordinates in $\CC^2$. We consider the
free $\cD'$-module with basis $\mathcal{B}_{k-1} :=\{
\partial_1^i :\; i=0,1,\ldots ,k-1\}$:  $$(\cD ')^{r +1} =
(\cD')^{k}\simeq\bigoplus_{i=0}^{k-1} \cD' \partial_1^i.$$

To apply  \cite[Algorithm 5.2.8]{SST} to our case we will use the
elements with $\omega$--order less than or equal to $k-1$, in the
Groebner basis $\cG:=\{P_1,P_2,P_3,E,R\}$ of $H_A(\beta)$. Recall
that $\cG$ is given by Lemma \ref{BGdeHAb} for all but finitely many
$\beta\in \CC$. Each operator $\partial_1^i P_1$, \; $\partial_1^i
E$, $i=0,\ldots , k-1$, must be written as a $\CC$--linear
combination of monomials $x^u
\partial^v$ and then substitute  $x_1=0$ into this expression. The result
is an element of  $(\cD ')^{k}=\cD' \mathcal{B}_k$. In this case we
get:
$$(\partial_1^i P_1)_{|x_1=0} = P_1 \partial_1^i, \; (\partial_1^i E)_{|x_1=0}= (k
a x_2
\partial_2 + k b x_3
\partial_3 -\beta + i)\partial_1^i, \; i=0,\ldots , k-1$$
and this proves the theorem.
\end{proof}

\begin{remark}\label{CK-gevrey} Let us consider  $A=(1\; a_2 \; \cdots
\; a_n )$, $1< a_2 < \cdots  < a_n $,
$k={\operatorname{gcd}}(a_{n-1} ,a_n )$ and  $A'= \frac{1}{k}
(a_{n-1} , a_n )$. We can apply  Cauchy-Kovalevskaya  Theorem for
Gevrey series (see \cite[Cor. 2.2.4]{Laurent-Mebkhout2}), Corollary
\ref{sumadirectan} and \cite[Prop. 4.2]{Castro-Takayama} to the
hypergeometric system $\cM_A(\beta)$ to prove that for each
$\beta\in \CC$ there exist $\beta_i\in \CC$, $i=0,\ldots,k-1,$ and a
quasi-isomorphism
$$\mathbb{R}\cH om_{\cD_{X}}(\HHAb , \cO_{\widehat{X|Y}}(s))_{|X'} \stackrel{\simeq}
{\longrightarrow} \bigoplus_{i=0}^{k-1} \mathbb{R}\cH
om_{\cD_{X'}}(\mathcal{M}_{A'}(\beta_i ) ,
\cO_{\widehat{X'|Y'}}(s))$$ for all $1\leq s \leq \infty$ where
$Y=(x_n = 0)$, $X'= (x_1 = x_2 = \cdots = x_{n-2}=0)$ and $Y'=Y\cap
X'$. Moreover, for all but finitely many $\beta$ we can take
$\beta_i= \frac{\beta -i}{k}$. Notice that coordinates in $X$, $Y$,
$X'$, $Y'$ are $x=(x_1 , \ldots ,x_{n})$, $y=(x_1 , \ldots
,x_{n-1})$, $x'=(x_{n-1} , x_{n})$ and $y'=(x_{n-1})$ respectively.

The last quasi-isomorphism induces a $\CC$--linear isomorphism
$$\mathcal{E} xt^j_{\cD_{X}}(\HHAb , \cO_{\widehat{X|Y}}(s))_{(0,\ldots ,0, \epsilon_{n-1} ,0)}
\stackrel{\simeq}{\longrightarrow}  \bigoplus_{i=0}^{k-1}
\mathcal{E} xt^j_{\cD_{X'}}(\mathcal{M}_{A'}(\beta_i ) ,
\cO_{\widehat{X'|Y'}}(s)))_{(\epsilon_{n-1 } ,0 )}$$ for all
$\epsilon_{n-1}\in \C$, $s\geq 1$ and $j\in \N$ and we also have
equivalent results for $\cQ_Y (s)$ and  $\cQ_{Y'}(s)$ instead of
$\cO_{\widehat{X|Y}}(s)$ and $\cO_{\widehat{X'|Y'}}(s)$.
\end{remark}

In particular, using  \cite[Proposition
5.9]{fernandez-castro-dim2-2008}, we have:

\begin{proposition}\label{ExtAn}
Let  $A=(1 \; a_2 \; \cdots \; a_n )$ be an integer row matrix with
$1<a_2 <\cdots < a_n $. Then for all $\beta \in \C$
$$\dim_{\C} (\mathcal{E} xt^j_{\cD_{X}}(\HHAb ,
\cQ_{Y}(s))_{(0,\ldots ,0, \epsilon_{n-1} ,0)})=\left\{
\begin{array}{ll}
a_{n-1} & \mbox{ if } s\geq a_n /a_{n-1}, \; j=0  {\mbox{ and }}  \epsilon_{n-1} \neq 0 \\
0 & \mbox{ otherwise }
\end{array}
\right.$$
\end{proposition}

\begin{corollary}\label{corciclo}
Let $A=(1 \; a_2 \; \cdots \; a_n )$ be an integer row matrix with
$1<a_2 <\cdots < a_n $. Then for all $ \beta \in \C$
$${\rm Ch}(\Irr_Y^{(s)}(\HHAb ))\subseteq T^{\ast}_Y Y \bigcup
T^{\ast}_{Z \cap Y} Y$$ for $s\geq \frac{a_{n}}{a_{n-1}}.$
\end{corollary}

\begin{proof} Here ${\rm Ch}(\Irr_Y^{(s)}(\HHAb ))$ is the
characteristic cycle of the perverse sheaf $\Irr_Y^{(s)}(\HHAb )$
(see e.g. \cite[Sec. 2.4]{Laurent-Mebkhout}). The Corollary follows
from the inclusion $${\rm Ch}^{(s)}(\cM_A(\beta)) \subset T_X^* X
\cup T^*_Y X \cup T^*_Z X$$ for $s\geq a_n/a_{n-1}$ and then by
applying \cite[Prop. 2.4.1]{Laurent-Mebkhout}.
\end{proof}

\begin{proof}(\textbf{Theorem \ref{teorext}})
Let us consider the Whitney stratification $Y=Y_1 \cup Y_2 $ of
$Y=(x_{n}=0) \subset  \C^{n}$ defined as
$$Y_1 := Y\setminus (Y\cap Z )\equiv
\C^{n-2} \times \C^{\ast}.$$
$$Y_2:= Y\cap Z \equiv \C^{n-2} \times \{
0\}.$$

As $\Irr_Y^{(s)}(\cM_A(\beta))$ is a perverse sheaf for all
$\beta\in \CC$ and $1\leq s \leq \infty$ \cite[Th. 6.3.3]{Mebkhout}
we can apply Kashiwara's constructibility Theorem
\cite{kashiwara-overdet-75}, the Riemann-Hilbert correspondence (see
\cite{Mebkhout-equiv-cat} and \cite{kashiwara-kawai-III},
\cite{kashiwara-rims-84}) and Corollary \ref{corciclo}, to prove
that
$$\mathcal{E}xt^i_\cD (\HHAb , \cQ_Y (s ) )_{|Y_j}
$$ is a locally constant sheaf of finite rank for all  $i\in \N$, $j=1,2$. To finish the proof it is enough to
apply Proposition \ref{ExtAn}.
\end{proof}

\begin{remark} Last proof uses Mebkhout's result on the perversity of the irregularity,
Kashiwara's constructibility Theorem and the Riemann-Hilbert
correspondence, all of them being deep results in $\cD$--module
theory. It would be interesting to give a more elementary proof of
Theorem \ref{teorext}. In \cite{fernandez-castro-dim2-2008} such an
elementary proof is given for the case $n=2$.
\end{remark}

\subsection{Gevrey solutions of $\cM_A(\beta)$} We will compute a
basis of the vector spaces $\mathcal{E}xt^i (\HHAb ,\cQ_Y(s))_p$ for
all  $ p\in Y$, $1\leq s \leq \infty$, $\beta \in \NN$, $i\in \NN$
and $A=(1 \; a_2 \; \cdots \; a_n )$ is an integer row matrix with
$1<a_2 <\cdots < a_n $. In fact it is enough to achieve the
computation  for $i=0$ and $p\in Y\setminus Z$, otherwise the
corresponding germ is zero by Theorem \ref{teorext}.

\begin{lemma}\label{exponents-smooth-curve} Let $A=(1 \; a_2 \; \cdots \; a_n )$ be
an integer row matrix with $1<a_2 <\cdots < a_n $ and $\omega \in
\R^n_{>0}$ satisfying
\begin{enumerate} \item[a)] $w_i
> a_i \omega_1\; {\mbox{ for  }} 2\leq i \leq n-2 \; {\mbox{ or }}
i=n$ \item [b)] $a_{n-1}\omega_1
> \omega_{n-1}$ \item [c)] $\omega_{n-1} > \omega_1 ,\ldots ,
\omega_{n-2}$ \end{enumerate} Then  $\HAb$ has $a_{n-1}$ exponents
with respect to $\omega$ and they have the form
$$v^{j}=(j , 0 , \ldots , 0 ,\frac{\beta -j}{a_{n-1}} ,0)\in \C^n$$
$j=0,1, \ldots , a_{n-1} -1$.
\end{lemma}

\begin{proof} The notion of {\em exponent} is given in \cite[page
92]{SST}. The toric ideal $I_A$ is generated by
$P_{1,i}=\partial_1^{a_i}-\partial_i \in \C [\partial ]$,
$i=2,\ldots , n$.

Let  $\omega= (\omega_1 ,\ldots ,\omega_n ) \in \R_{>0}^n$ be a
weight vector satisfying  the statement of the lemma. We have:

$$\inww P_{1,i} = \left
\{ \begin{array}{ll}
\partial_i  & \mbox{ if } i=2,\ldots ,n-2,n \\
\partial_1^{a_{n-1}} & \mbox{ if } i=n-1
\end{array}
\right.$$

In particular $\{P_{1,i}:\; i=2,\ldots ,n \}$ is a  Groebner basis
of $I_A$ with respect to  $(-\omega,\omega)$ and then

$$\inw I_A = \langle  \partial_2 ,\ldots ,\partial_{n-2}, \partial_1^{a_{n-1}}, \partial_n
\rangle.$$

The standard  pairs of  $\inw ( I_A )$ are (\cite[Sec. 3.2]{SST}):

$$\mathcal{S}(\inw ( I_A ))=\{ (\partial_1^j , \{ n-1 \} ): \; j=0,1,\ldots ,a_{n-1}-1
\}$$

To the standard pair  $(\partial_1^j , \{ n-1 \} )$ we associate,
following \cite[page 108]{SST}, the {\em fake exponent}
$$v^{j}=(j,0,\ldots ,0,\frac{\beta -j}{a_{n-1}} ,0)$$ of the module
$\HHAb$ with respect to $\omega$. It is easy to prove that these
fake exponents are in fact exponents since they have minimal
negative support \cite[Th. 3.4.13]{SST}. \end{proof}

\begin{remark}\label{remark-exponents-smooth-curve} With the above
notation,  the $\Gamma$--series $ \phi_{v^{j}}$ associated with
$v^j$ for $j=0,\ldots,a_{n-1}-1$, is defined as
$$\phi_{v^j} = x^{v^j} \sum_{u\in L_A} {\Gamma[v^j;u]} x^u $$
where $L_A=\ker_\ZZ(A)$ is the lattice generated by the vectors
$\{u^2,\ldots,u^n\}$ and  $u^i$ is the $(i-1)$-th row of the matrix
$$\left(\begin{array}{rcccccrc}
-a_2     & 1 & 0 & \cdots &0&0&0&0 \\
\vdots & \vdots &\vdots & \vdots &\vdots &\vdots &\vdots &\vdots \\
-a_{n-2} & 0 & 0 & \cdots &0&1&0&0 \\
a_{n-1}  & 0 & 0 & \cdots &0&0&-1&0 \\
-a_{n}   & 0 & 0 & \cdots &0&0&0&1
\end{array}\right).$$

For any ${\bf m}=(m_2,\ldots,m_n)\in \ZZ^{n-1}$ let us denote
$u({\bf m}):=\sum_{i=2}^n m_i u^i\in L_A$. We can write
$$\phi_{v^{j}} = x^{v^j}
\sum_{\stackrel{m_2,\ldots, m_{n-1},m_n \geq 0}{ _{\sum_{i\neq n-1}
a_i m_i \leq j+a_{n-1}m_{n-1}}}} \Gamma[v^j; u({\bf m})] x^{u({\bf
m})}$$ for $j=0,1,\ldots ,a_{n-1}-1$. We have for ${\bf
m}=(m_2,\ldots,m_{n}) \in \NN^{n-1}$ such that $j-\sum_{i\neq n-1}
a_i m_i + a_{n-1} m_{n-1} \geq 0$
$$\Gamma[v^j; u({\bf m})] =
\frac{(\frac{\beta -j}{a_{n-1}})_{m_{n-1}}j!} {m_2! \cdots m_{n-2}!
m_n ! (j-\sum_{i\neq n-1} a_i m_i + a_{n-1} m_{n-1})!}$$ and
$$x^{u({\bf m})} = x_1^{-\sum_{i\neq n-1} a_i m_i + a_{n-1}
m_{n-1}} x_2^{m_2} \cdots x_{n-2}^{m_{n-2}} x_{n-1}^{-m_{n-1}}
x_n^{m_n}$$
\end{remark}

\begin{proposition}\label{propo-ext0formal} Let $A=(1 \; a_2 \; \cdots \; a_n )$ be an
integer row matrix with $1<a_2 <\cdots < a_n $, $Y=(x_n=0)\subset X$
and $Z=(x_{n-1}=0)\subset X$. Then we have:
$$\mathcal{E}xt^0 (\HHAb ,\cO_{\widehat{X|Y}})_p
= \sum_{j=0}^{a_{n-1}-1} \CC  \phi_{v^{j},p}$$
for all  $\beta \in \CC$, $p\in Y\setminus Z$.
\end{proposition}
\begin{proof} {\em Step 1.-} Using \cite{GZK} and \cite{SST} we will
describe $a_{n-1}$ linearly independent solutions of
$\cM_A(\beta)_p$, living in some Nilsson series ring. Then, using
initial ideals, we will prove that an upper bound of the dimension
of $\cE xt_{}^0(\cM_A(\beta),\cO_{\widehat{X|Y}})_p$ is $a_{n-1}$
for $p$ in $Y\setminus Z$.


First of all, the series $$\{\phi_{v^{j}}\,\vert\, j=0,\ldots
,a_{n-1}-1\} \subset x^{v^j}\CC[[x_1^{\pm
1},x_2,\ldots,x_{n-2},x_{n-1}^{-1},x_n]]$$ described in Remark
\ref{remark-exponents-smooth-curve}, are linearly independent since
$\inw (\phi_{v^{j}})=x^{v^{j}}$ for $0\leq j\leq a_{n-1}-1$. They
are solutions of the system $\cM_A(\beta)$ (see \cite{GGZ},
\cite[Sec. 1]{GZK},\cite[Sec. 3.4]{SST}).

On the other hand \begin{align}\label{dim-menor-igual-an-1}
\dim_{\C}\mathcal{E}xt^0 (\HHAb , \cO_{\widehat{X|Y}})_p \leq
a_{n-1}\end{align} for $p=(\epsilon_1 ,\ldots , \epsilon_{n-1},0)$,
$\epsilon_{n-1}\neq 0$. This follows from the following facts: \\ a)
The initial ideal $\inw (I_A )$ equals  $\langle
\partial_2 ,\ldots ,\partial_{n-2},
\partial_1^{a_{n-1}},
\partial_n \rangle$  for  $\omega$ as in  Lemma \ref{exponents-smooth-curve}. \\ b) The germ of
$E$ at $p$ is nothing but $E_{p}:=E+\sum_{i=1}^{n-1} a_i \epsilon_i
\partial_i $ (here $a_1=1$) and satisfies $$\inww (E_p )=
a_{n-1}\epsilon_{n-1} \partial_{n-1}.$$ c) By \cite[Th. 2.5.5]{SST}
if $f\in \cO_{\widehat{X|Y},p} $ is a solution of the ideal
$H_A(\beta)$ then $\inw(f)$ must the annihilated by $\inww
(H_A(\beta))$. That proves inequality (\ref{dim-menor-igual-an-1}).

{{\em Step 2.-}} It is easy to prove, using standard estimates, that
the series $\phi_{v^{j},p}$ are in fact in $\cO_{\widehat{X|Y},p}$
for all  $p\in Y\setminus Z$. Then by step 1, they form a basis of
the  the vector space $\mathcal{E}xt^0 (\HHAb ,
\cO_{\widehat{X|Y}})_p $ for $p\in Y\setminus Z$.
\end{proof}

\begin{remark} We will prove (see Theorem \ref{ext0formal}) that the Gevrey index of $\phi_{v^j,p}$
is $a_n/a_{n-1}$ for $\beta\in \CC$ and $p\in Y\setminus Z$ except
for $\beta\in \NN$ and $j=q$ the unique integer $0\leq q \leq
a_{n-1}-1$ such that $\beta -q\in a_{n-1}\NN$. In that case
$\phi_{v^q}$ is a polynomial (see details in Remark
\ref{phi_w-bis}).
\end{remark}

\begin{proposition}\label{ext-j-convergent} Let $A=(1 \; a_2 \; \cdots \; a_n )$ be an
integer row matrix with $1<a_2 <\cdots < a_n $, $Y=(x_n=0)\subset
X$. Then we have: \begin{enumerate}
\item[i)]
For each $\beta \in \CC\setminus \NN$ we have $\mathcal{E}xt^j
(\HHAb ,\cO_{{X|Y}}) =0$ for all $j\in \NN$.
\item[ii)] For each $\beta \in \NN$ the sheaf $\mathcal{E}xt^j (\HHAb ,\cO_{{X|Y}})$
is locally constant on $Y$ of rank 1 for $j=0,1$ and it is zero for $j\geq 2$. \end{enumerate}
\end{proposition}

\begin{proof} For all  $\beta \in \C$ the
characteristic variety of $\HHAb$ is ${\rm Ch}(\HHAb )= T_X^{\ast} X
\cup T_Y^{\ast} X$ (see e.g.  \cite{Adolphson}). Then from
Kashiwara's constructibility Theorem \cite{kashiwara-overdet-75}
we have  that, for all $j\in \N$, the sheaf
\begin{align}
\mathcal{E}xt^j (\HHAb , \cO_X )_{|Y} = \mathcal{E}xt^j (\HHAb ,
\cO_{X|Y} )_{|Y} \label{hazloccte}
\end{align}
is locally constant.

Assume $\beta \notin \N$. From Corollary  \ref{sumadirectan} and
Proposition \ref{M-a-beta-isom} we have  that  there exists $m\in
\N$ such that $\HHAb \simeq \mathcal{M}_A (\beta - m)$ and
$$\HHAb_{|X'}\simeq \mathcal{M}_A (\beta - m)_{|X'}\simeq
\bigoplus_{i=0}^{k-1} \mathcal{M}_{(a'_{n-1} \; a'_n )}
\left(\frac{\beta - m -i}{k} \right) $$ with  $X'=(x_1 =\cdots =
x_{n-2}=0)\subset X$, $k=\operatorname{gcd}(a_{n-1}, a_n )$ and
$a'_{\ell}=\frac{a_{\ell}}{k}$ for $\ell=n-1,n$. Then by applying
Cauchy-Kovalevskaya Theorem (see Remark \ref{CK-gevrey}) we get:
$$\mathcal{E}xt^j (\mathcal{M}_A(\beta), \cO_{X|Y})_{|X'}
\simeq \bigoplus_{i=0}^{k-1}\mathcal{E}xt^j (\mathcal{M}_{(a'_{n-1}
\; a'_n )} \left(\frac{\beta - m -i}{k} \right) , \cO_{X' | Y'} )$$
with $Y'=X'\cap Y$.

As $\beta \notin \N$ then $\frac{\beta - m -i}{k} \notin a'_{n-1}\N
+a'_{n}\N$ for  $i=0 ,\ldots , k-1$. Then part {\em i)} follows from
\cite[Proposition 4.5]{fernandez-castro-dim2-2008}.

Assume now $\beta\in \NN$. From Corollary  \ref{sumadirectan} and
Proposition \ref{M-a-beta-isom} we have that  there exists $m\in \N$
such that $\HHAb \simeq \mathcal{M}_A (\beta + m)$ and
$$\HHAb_{|X'}\simeq \mathcal{M}_A (\beta + m)_{|X'}\simeq
\bigoplus_{i=0}^{k-1} \mathcal{M}_{(a'_{n-1} \; a'_n )}
\left(\frac{\beta + m -i}{k} \right).$$ Then by applying
Cauchy-Kovalevskaya Theorem (see Remark \ref{CK-gevrey}) we get:
$$\mathcal{E}xt^j (\mathcal{M}_A(\beta), \cO_{X|Y})_{|X'} \simeq
\bigoplus_{i=0}^{k-1}\mathcal{E}xt^j (\mathcal{M}_{(a'_{n-1} \; a'_n
)} \left(\frac{\beta + m -i}{k} \right) , \cO_{X' | Y'}).$$ By
\cite{fernandez-castro-dim2-2008} this last module is in fact equal
to $$\mathcal{E}xt^j (\mathcal{M}_{(a'_{n-1} \; a'_n )}
\left(\frac{\beta + m -i_0}{k} \right) , \cO_{X' | Y'})$$ where
$i_0$ is the unique integer number such that $0\leq i_0\leq k-1$ and
$\beta + m-i_0\in k \NN$. Then part {\em ii)} follows from
\cite[Proposition 4.6]{fernandez-castro-dim2-2008}.
%
\end{proof}

\begin{theorem}\label{ext0formal} Let $A=(1 \; a_2 \; \cdots \; a_n )$ be an
integer row matrix with $1<a_2 <\cdots < a_n $, $Y=(x_n=0)\subset X$
and $Z=(x_{n-1}=0)\subset X$. Then we have:
\begin{enumerate} \item[i)]
$$\mathcal{E}xt^0 (\HHAb ,\cO_{\widehat{X|Y}}(s))_p
= \bigoplus_{j=0}^{a_{n-1}-1} \CC  \phi_{v^{j},p}$$
for all  $\beta \in \CC$, $p\in Y\setminus Z$ and $s\geq a_n
/a_{n-1}.$
\item[ii)] $$\mathcal{E}xt^0 (\HHAb ,\cO_{\widehat{X|Y}}(s))_p =
\left\{\begin{array}{cc}
0 & \mbox{ if } \beta \notin \N \\
\C \phi_{v^q} & \mbox{ if  } \beta \in \N
\end{array}\right.$$
for all $p\in Y\setminus Z$ and $1\leq s < a_n /a_{n-1}$, where $q$
is the unique element in  $\{0,1,\ldots , a_{n-1} -1 \}$ satisfying
$\frac{\beta -q}{a_{n-1}} \in \N$ and $\phi_{v^q}$ is a polynomial.
\end{enumerate}
\end{theorem}

\begin{proof} {\em i)} Let us consider $a_n/a_{n-1} \leq s \leq
\infty$ and $p\in Y\setminus Z$. Assume first that $\beta\not\in
\NN$. By Proposition \ref{ext-j-convergent} and the long exact
sequence of cohomology associated with the short exact sequence
(\ref{sucs}) we have that
$$\cE xt^0(\cM_A(\beta), \cO_{\widehat{X|Y}}(s))_p \simeq \cE
xt^0(\cM_A(\beta), \cQ_Y(s))_p
$$ and by Theorem \ref{teorext} this last vector space has dimension
$a_{n-1}$. As
$$\cE
xt^0(\cM_A(\beta), \cO_{\widehat{X|Y}}(s))_p \subset  \cE
xt^0(\cM_A(\beta), \cO_{\widehat{X|Y}})_p
$$ part {\em i)} follows from  Proposition
\ref{propo-ext0formal} if $\beta \not\in \NN$.

Assume now $\beta\in \NN$. Applying the long exact sequence of
cohomology associated with the short exact sequence (\ref{sucs}),
Theorem \ref{teorext} and Proposition \ref{ext-j-convergent} we get
an exact sequence of vector spaces
$$0 \rightarrow \CC \rightarrow \cL_1 \rightarrow \cE
xt^0(\cM_A(\beta), \cQ_Y(s))_p \rightarrow \CC \rightarrow \cL_2
\rightarrow 0$$ where $\cL_1 = \cE xt^0(\cM_A(\beta),
\cO_{\widehat{X|Y}}(s))_p$ and $\cL_2=\cE xt^1(\cM_A(\beta),
\cO_{\widehat{X|Y}}(s))_p$. Let us write $\nu_i=\dim_\CC(\cL_i)$. By
Theorem \ref{teorext} we also have $\nu_1=a_{n-1}+\nu_2$. On the
other hand, by Proposition \ref{propo-ext0formal}, we know that
$\nu_1 \leq a_{n-1}$. This implies $\nu_1=a_{n-1}$ and
$\cL_2=\{0\}$. In particular we have the equality
$$\cE
xt^0(\cM_A(\beta), \cO_{\widehat{X|Y}}(s))_p \subset  \cE
xt^0(\cM_A(\beta), \cO_{\widehat{X|Y}})_p
$$ part {\em i)} also follows from  Proposition
\ref{propo-ext0formal} if $\beta \in \NN$. \\ Let us prove part {\em
ii)}. First of all, by Theorem \ref{teorext}, $\cE
xt^j(\cM_A(\beta), \cQ_Y(s))_p=0$ for all $j\in \NN$. Then the
result follows from  the long exact sequence of cohomology
associated with the short exact sequence (\ref{sucs}) and
Proposition \ref{ext-j-convergent}.
\end{proof}

\begin{remark}\label{phi_w-bis}
Let us recall here the notations introduced in Lemma
\ref{exponents-smooth-curve}. For $A=(1 \; a_2 \; \cdots \; a_n )$
an integer row matrix with $1<a_2 <\cdots < a_n $ and $\omega \in
\R^n_{>0}$ satisfying
\begin{enumerate} \item $w_i
> a_i \omega_1$, for  $2\leq i \leq n-2$ or $i=n$ \item  $a_{n-1}\omega_1 >
\omega_{n-1} $ \item  $\omega_{n-1} > \omega_1 ,\ldots ,
\omega_{n-2}$\end{enumerate} we have proved (see Lemma
\ref{exponents-smooth-curve}) that $\HAb$ has $a_{n-1}$ exponents
with respect to  $\omega$ and that they have the form:
$$v^{j}=(j , 0 , \ldots , 0 ,\frac{\beta -j}{a_{n-1}} ,0)\in \C^n$$
$j=0,1, \ldots , a_{n-1} -1$.

The corresponding $\Gamma$--series $ \phi_{v^{j}}$ is defined as:

$$\phi_{v^{j}} = x^{v^j}
\sum_{\stackrel{m_2,\ldots, m_{n-1},m_n \geq 0}{ _{\sum_{i\neq n-1}
a_i m_i \leq j+a_{n-1}m_{n-1}}}} \Gamma[v^j; u({\bf m})] x^{u({\bf
m})}$$ for $j=0,1,\ldots ,a_{n-1}-1$, where for any ${\bf
m}=(m_2,\ldots,m_n)\in \ZZ^{n-1}$ we  denote $u({\bf
m}):=\sum_{i=2}^n m_i u^i\in L_A$.

For ${\bf m}=(m_2,\ldots,m_{n}) \in \NN^{n-1}$ such that
$j-\sum_{i\neq n-1} a_i m_i + a_{n-1} m_{n-1} \geq 0$, we have
$$\Gamma[v^j; u({\bf m})] = \frac{(\frac{\beta
-j}{a_{n-1}})_{m_{n-1}}j!} {m_2! \cdots m_{n-2}! m_n !
(j-\sum_{i\neq n-1} a_i m_i + a_{n-1} m_{n-1})!}$$ and
$$x^{u({\bf m})} = x_1^{-\sum_{i\neq n-1} a_i m_i + a_{n-1}
m_{n-1}} x_2^{m_2} \cdots x_{n-2}^{m_{n-2}} x_{n-1}^{-m_{n-1}}
x_n^{m_n}.$$

If $\beta\in \NN$ then there exists a unique $0\leq q \leq
a_{n-1}-1$ such that $\beta - q \in a_{n-1}\NN$. Let us write
$m_0=\frac{\beta -q}{a_{n-1}}$.

Then for $m\in \NN$ big enough $m_0-a_n m$ is a negative integer and
the coefficient ${\Gamma[v^q; u({\bf m})]}$ is zero and then
$\phi_{v^q}$ is a polynomial in $\CC[x]$.

Recall that $u^{n-1}=(a_{n-1},0,\ldots,-1,0)\in L_A$ and let us
write $$\widetilde{v^q} = v^q + (m_0+1)u^{n-1} =
(q+(m_0+1)a_{n-1},0,\ldots,0,-1,0)=(\beta +
a_{n-1},0,\ldots,0,-1,0).$$ We have $A\widetilde{v^q} = \beta$ an
the corresponding $\Gamma$--series is
$$\phi_{\widetilde{v^q}} = x^{\widetilde{v^q}} \sum_{{\bf m}\in M(q)}
\Gamma[\widetilde{v^q};u({\bf m})] x^{u({\bf m})} $$ where for ${\bf
m} = (m_2,\ldots,m_n)\in \ZZ^n$ one has $u({\bf m}) = \sum_{i=2}^n
m_i u^{i}$  and
$$M(q):=\{(m_2,\ldots,m_n)\in \NN^{n-1}\, \vert \, q+(m_0+m_{n-1}+1)a_{n-1} -\sum_{i\neq n-1} a_i m_i \geq 0\}.$$

Let us notice that $\widetilde{v^q}$ does not have minimal negative
support (see \cite[p. 132-133]{SST}) and then the $\Gamma$--series
$\phi_{\widetilde{v^q}}$ is not a solution of $H_A(\beta)$.  We will
prove in Theorem \ref{basis_of_ext_i_Q_s} that
$H_A(\beta)_p(\phi_{\widetilde{v^q},p}) \subset \cO_{X,p}$ for all
$p\in Y\setminus Z$ and that $\phi_{\widetilde{v^q},p}$ is a Gevrey
series of index $a_{n}/a_{n-1}$.
\end{remark}

The second main result of this Section is the following
\begin{theorem}\label{basis_of_ext_i_Q_s}
Let $A=(1 \; a_2 \; \cdots \; a_n )$ be an integer row matrix with
$1<a_2 <\cdots < a_n $, $Y=(x_n=0)\subset X$ and
$Z=(x_{n-1}=0)\subset X$. Then for all $p\in Y\setminus Z$ and
$s\geq a_n /a_{n-1}$ we have:
\begin{enumerate}
\item[i)] If $\beta \notin \N$, then:
$$\mathcal{E}xt^0 (\HHAb ,\cQ_Y (s))_p = \bigoplus_{j=0}^{a_{n-1}-1}
\C  \overline{\phi_{v^{j} , p}}.$$
\item[ii)] If  $\beta \in \N$, then there exists a unique
$q\in \{0,\ldots , a_{n-1} -1 \}$ such that $\frac{\beta -
q}{a_{n-1}}\in \N$
and we have:
$$\mathcal{E}xt^0 (\HHAb ,\cQ_Y (s))_p = \bigoplus_{q\neq j=0}^{a_{n-1}-1}
\C \overline{\phi_{v^{j},p }} \oplus  \C
\overline{\phi_{\widetilde{v^q},p}}.$$
\end{enumerate}
Here  $\overline {\phi}$ stands for the class modulo $\cO_{X|Y,p}$
of  $\phi \in \cO_{\widehat{X|Y} ,p}(s)$.
\end{theorem}

\begin{proof} Part
{\em i)} follows from Theorem  \ref{ext0formal} and Proposition
\ref{ext-j-convergent} using the long exact sequence of cohomology.

Let us prove  {\em ii).}
Since  $\mathcal{E}xt^1 (\HHAb ,\cQ_Y (s))=0$ (see Theorem
\ref{teorext}) and applying Theorem \ref{ext0formal}, Proposition
\ref{ext-j-convergent}  and the long exact sequence in cohomology we
get that
$$\mathcal{E}xt^1 (\HHAb ,\cO_{\widehat{X|Y}} (s))_{|Y\setminus Z}$$ is zero for
$s\geq a_n /a_{n-1}$ and locally constant of rank 1 for $1\leq s<
a_n /a_{n-1}$. We also have that $$\mathcal{E}xt^1 (\HHAb
,\cO_{\widehat{X|Y}} (s))_{|Y\cap Z}$$ is locally constant of rank 1
for all $s\geq 1$.

Assume  $s\geq a_n /a_{n-1}$. We consider the following  long exact
sequence associated with the short exact sequence \ref{sucs} (with
$p\in  Y\setminus Z$ and $\cM=\cM_A(\beta)$) {\footnotesize
$$0\rightarrow \mathcal{E}xt^0 (\cM ,\cO_{X|Y} )_{p}\rightarrow
\mathcal{E}xt^0 (\cM ,\cO_{\widehat{X|Y}}(s) )_{p}
\stackrel{\rho}{\rightarrow} \mathcal{E}xt^0 (\cM ,\cQ_Y (s) )_{p}
\rightarrow \mathcal{E}xt^1 (\cM ,\cO_{X|Y})_{p} \rightarrow 0$$} We
also have
$$ \mathcal{E}xt^0 (\HHAb ,\cO_{X|Y} )_{p} \simeq \C$$
$$\mathcal{E}xt^0 (\HHAb ,\cO_{\widehat{X|Y}}(s) )_{p}\simeq
\C^{a_{n-1}}$$
$$\mathcal{E}xt^0 (\HHAb ,\cQ_Y (s) )_{p}\simeq \C^{a_{n-1}}$$
$$\mathcal{E}xt^1 (\HHAb ,\cO_{X|Y})_{p}\simeq \C$$

Since $\beta \in \N$ there exists a unique $q=0,1,\ldots , a_{n-1}
-1$ such that $\frac{\beta -q}{a_{n-1}}\in \N$  and then
$\phi_{v^q}\in \C [x]$ generates $\mathcal{E}xt^0 (\HHAb ,\cO_{X|Y}
)_{p}=\operatorname{Ker}(\rho )$.

Using
the above exact sequence and the first isomorphism theorem we get
that the family $$\{\overline{\phi_{v^j ,p }} : \; 0\leq j\leq
a_{n-1}-1, j\neq q \}$$  is linearly independent in $\cQ_Y (s)_p$
for all $p \in Y\setminus Z$.

In a similar way to the proof of Theorem \ref{teorext} it can be
proved that  ${\phi}_{\widetilde{v^{q}},p} \in
\cO_{\widehat{X|Y}}(s)_p $ for all $p\in Y\setminus Z$ and $s \geq
a_n/a_{n-1}$.

Writing $t_{n-1}=x_{n-1}^{-1}$ and defining:
$${\psi}_{\widetilde{v^q}}(x_1 ,\ldots ,x_{n-2} ,t_{n-1},x_n):=
x^{-\widetilde{v^q}}{\phi}_{\widetilde{v^q}}(x_1 ,\ldots ,x_{n-2}
,\frac{1}{t_{n-1}},x_n)
$$ we have that $${\psi}_{\widetilde{v^q}} \in \C [[x_1 ,\ldots , x_{n-2}
,t_{n-1}, x_n ]]$$

Taking the subsum of ${\psi}_{\widetilde{v^q}}$ for  $m_2=\cdots =
m_{n-2}=0 , \; m_n= a_{n-1} m, \;  m_{n-1}=a_n m, m\in \N$, we get
the power series
$$ \sum_{m\geq 0} c_m \left(t_{n-1}^{a_n} x_n^{a_{n-1}}\right)^m $$
where $$c_m= \frac{(-1)^{a_nm} (a_n m)!}{(a_{n-1}m)!}$$

This power series has Gevrey index $a_n /a_{n-1}$ with respect to
$x_n=0$. Then ${\phi_{\widetilde{v^q}}}$ has Gevrey index $a_n
/a_{n-1}$.

We have $E({\phi}_{\widetilde{v^q}})=P_i
({\phi}_{\widetilde{v^q}})=0$, for all $i=1,2,\ldots ,n-2,n$ and
$P_{n-1}({\phi}_{\widetilde{v^q}})$ is a meromorphic  function with
poles along $Z$ (and holomorphic on $X\setminus Z$):

$$P_{n-1} ({\phi}_{\widetilde{v^q}}) =
\sum_{\underline{m}\in \widetilde{M}(q)} \frac{(\beta + a_{n-1})!
x_1^{q-\sum_{i\neq n-1} a_i m_i + a_{n-1} (m_0 +1)} x_2^{m_2} \cdots
x_{n-2}^{m_{n-2}} x_{n-1}^{-1} x_n^{m_n}}{m_2! \cdots m_{n-2}! m_n !
(q-\sum_{i\neq n-1} a_i m_i + a_{n-1} (m_0 +1))!}$$ where
$$\widetilde{M}(q) =\{(m_2
,\ldots , m_{n-2} ,m_n) \in \N^{n-2} \vert  \; \sum a_i m_i \leq
q+a_{n-1}(m_0 +1)=\beta + a_{n-1} \}$$ is a finite set (recall that
$m_0 =\frac{\beta -q}{a_{n-1}}\in \N$).

In particular, $\HAb ({\phi}_{\widetilde{v^q}}) \subseteq
\cO_{X}(X\setminus Z)$.

So, $$\overline{{\phi}_{\widetilde{v^q},p}} \in \mathcal{E}xt^0
(\HHAb , \cQ_Y (s))_p$$ for all $p\in Y\setminus Z$ and $s\geq
a_n/a_{n-1}$.

In order to finish the proof we will see that for all $ \lambda_j
\in \C$ ($j=0,\ldots, a_{n-1}-1$; $j\not= q$) and for all  $ p\in
Y\setminus Z$ we have $${\phi}_{\widetilde{v^q}, p}-\sum_{j\neq q}
\lambda_j \phi_{v^j ,p} \notin \cO_{X|Y, p}.$$

Let us write $$\psi_{v^j}(x_1 ,\ldots ,x_{n-2} ,t_{n-1},x_n):=
\phi_{v^{j}}(x_1 ,\ldots ,x_{n-2} ,\frac{1}{t_{n-1}},x_n)  $$

Assume to the contrary that there exist $p\in Y\setminus Z$ and
$\lambda_j \in \C$ such that:
$${\phi}_{\widetilde{v^q}, p}-\sum_{j\neq q} \lambda_j \phi_{v^j ,p}
\in \cO_{X|Y, p}$$

Let us consider the holomorphic function   at $p$ defined as $$f:=
x^{\widetilde{v^q}}{\psi}_{\widetilde{v^q}, p}-\sum_{j\neq q}
\lambda_j \psi_{v^j ,p}$$

We have the following equality of holomorphic functions at $p$:
$$\rho_s (f +\sum_{j\neq q} \lambda_j \psi_{v^j}) = \rho_s (x^{\widetilde{v^q}}{\psi}_{\widetilde{v^q}})$$
for  $s> a_n$.

The function $ \rho_s (x^{\widetilde{v^q}}{\psi}_{\widetilde{v^q}})$
is holomorphic in $\C^n$ while each $\rho_s (\psi_{v^j})$ has the
form $t_{n-1}^{-\frac{\beta -j}{a_{n-1}}}\psi_j$ with $\psi_j$
holomorphic in $\C^n$.

Making a loop around the $t_{n-1}$ axis  ($\log t_{n-1} \mapsto \log
t_{n-1} +2\pi i $) we get the equality:

$$\rho_s (\widehat{f} +\sum_{j\neq q} c_j \lambda_j \psi_{v^j}) = \rho_s
(x^{\widetilde{v^q}}{\psi}_{\widetilde{v^q}})$$ where  $c_j
=e^{-\frac{\beta -j}{a_{n-1}}2\pi i}\neq 1$ (since $\frac{\beta
-j}{a_{n-1}}\notin \Z$  for all $j\neq q$) and  $\widehat{f}$ is
obtained from $f$ after the loop. Since $f$ is holomorphic at $p$
then $\widehat{f}$ also is.  Subtracting both equalities we get:

$$\rho_s (\widehat{f}-f +\sum_{j\neq q} (c_j -1) \lambda_j \psi_{v^j}) = 0
$$ and then
$$\sum_{j\neq q} (c_j -1) \lambda_j \psi_{v^j}= f - \widehat{f}
$$ in the neighborhood of $p$. This  contradicts the fact that the power series
$\{ \phi_{v^j}: \; j\neq q , 0\leq j\leq a_{n-1}-1 \}$ are linearly
independent modulo $\cO_{X|Y ,p }$ (here we have  $c_j -1 \neq 0$).
This proves the theorem.
\end{proof}

\begin{corollary}
If $\beta \in \N$ then for all $p\in Y\setminus Z$ the vector space
$\mathcal{E}xt^1 (\HHAb , \cO_{X|Y})_p$ is generated by the class
of: $$(P_2 ({\phi}_{\widetilde{v^q}}),\ldots , P_{n-1}
({\phi}_{\widetilde{v^q}}), P_n ({\phi}_{\widetilde{v^q}}),
E({\phi}_{\widetilde{v^q}}) ) = $$
$$= (0,\ldots ,0,\sum_{\underline{m}\in \widetilde{M}(q)}
\frac{(\beta + a_{n-1})!\, x_1^{q-\sum_{i\neq n-1} a_i m_i + a_{n-1}
(m_0 +1)} x_2^{m_2} \cdots x_{n-2}^{m_{n-2}} x_{n-1}^{-1}
x_n^{m_n}}{m_2! \cdots m_{n-2}! m_n ! (q-\sum_{i\neq n-1} a_i m_i +
a_{n-1} (m_0 +1))!},0 ,0)$$ in  $$\frac{(\cO_{X|Y})_p^{n}}{{\rm Im}
(\psi_0^{\ast} , \cO_{X|Y})_p}$$ where
$$\widetilde{M}(q) =\{(m_2 ,\ldots , m_{n-2} ,m_n
) \in \N^{n-2}  \; \vert \, \sum a_i m_i \leq q+a_{n-1}(m_0
+1)=\beta + a_{n-1} \}$$ is a finite set (with $m_0 =\frac{\beta
-q}{a_{n-1}}\in \N$) and $\psi_0^{\ast}$ being the dual map of
$$\begin{array}{rcl}
          \psi_0 : \cD^n  & \longrightarrow & \cD \\
           (Q_1 ,\ldots ,Q_n) & \mapsto & \sum_{j=2}^{n} Q_j P_j +
           Q_ n E
         \end{array}$$
\end{corollary}

\begin{proof} It follows from the proof of Theorem
\ref{basis_of_ext_i_Q_s} since  $\mathcal{E}xt^1 (\HHAb
,\cO_{X|Y})_{p}\simeq \C$ for all  $p\in Y\setminus Z$ and moreover
$$(P_2 (\phi_{v^j}),\ldots , P_n
(\phi_{v^j}), E(\phi_{v^j}) )= \underline{0}$$ for  $0\leq j \leq
a_{n-1}-1$, $j\neq q$.
\end{proof}


\begin{remark}\label{sol_generic_point_1a2an} We can also compute the holomorphic solutions
of $\cM_A(\beta)$ at any point in $X\setminus Y$ for $A=(1 \; a_2 \;
\ldots \; a_n)$ with $1 < a_2 < \ldots < a_n$ and for any $\beta \in
\CC$, where $Y=(x_n=0)\subset X=\CC^n$. We consider the vectors
$w^{j} = (j,0,\ldots ,0, \frac{\beta - j}{a_n })\in \CC^{n}$,
$j=0,1,\ldots ,a_n -1$ then the germs at $p \in X \setminus Y $ of
the series solutions $\{ \phi_{w^{j}}:\; j=0,1,\ldots ,a_n -1\}$ is
a basis of $\cE xt^i_{\cD}(\cM_{A}(\beta),\cO_X)_p$.
\end{remark}

We have summarized the main results of this Section in Figure 1.
Here $A=(1\; a_2 \; \cdots \; a_n )$, $s\geq  \frac{a_n}{a_{n-1}}$,
$p\in Y\setminus Z$, $z\in Y \cap Z $, $\beta_{{\rm esp}}\in \N$ and
$\beta_{{\rm gen}}\notin \N$.

\begin{figure}[h]
\begin{center}
\begin{tabular}{||c|c|p{1.2cm}|p{1.2cm}|p{1.2cm}|p{1.2cm}||}
    \hline
   $(z,\beta_{{\rm esp}})$ & $(p,\beta_{{\rm esp}})$  & \multicolumn{2}{c|}{}  &  \multicolumn{2}{c||}{}  \\  \cline{1-2}
    $(z,\beta_{{\rm gen}})$ & $(p,\beta_{{\rm gen}})$  & \multicolumn{2}{c|}{$\mathcal{E}xt^0 (\HHAb ,\cF)$}  &  \multicolumn{2}{c||}{$\mathcal{E}xt^1 (\HHAb ,\cF) $}  \\ \cline{1-6}
    \multicolumn{2}{||c|}{\multirow{2}{3cm}{ $\cF=\cO_{X|Y}$}} & $\; \; \; \; 1$ & \multicolumn{1}{c|}{1} & $\; \; \; \; 1$ & \multicolumn{1}{c||}{1} \\ \cline{3-6}
    \multicolumn{2}{||c|}{} & $\; \; \; \; 0$ & \multicolumn{1}{c|}{0} & $\; \; \; \; 0$ & \multicolumn{1}{c||}{0} \\  \cline{1-6}
    \multicolumn{2}{||c|}{\multirow{2}{3cm}{ $\cF=\cO_{\widehat{X|Y}}(s)$}} & $\; \; \; \; 1$ & \multicolumn{1}{c|}{$a_{n-1}$} & $\; \; \; \; 1$ & \multicolumn{1}{c||}{0} \\ \cline{3-6}
    \multicolumn{2}{||c|}{} & $\; \; \; \; 0$ & \multicolumn{1}{c|}{$a_{n-1}$} & $\; \; \; \; 0$ & \multicolumn{1}{c||}{0} \\  \cline{1-6}
    \multicolumn{2}{||c|}{\multirow{2}{3cm}{ $\cF=\cQ_Y (s)$}}  & $\; \; \; \; 0$ & \multicolumn{1}{c|}{$a_{n-1}$} & $\; \; \; \; 0$ & \multicolumn{1}{c||}{0} \\ \cline{3-6}
    \multicolumn{2}{||c|}{} & $\; \; \; \; 0$ & \multicolumn{1}{c|}{$a_{n-1}$} & $\; \; \; \; 0$ & \multicolumn{1}{c||}{0} \\  \cline{1-6}
    \hline
  \end{tabular} \caption{Dimension of the germs of $\cE
  xt^i_{\cD_X}(\cM_A(\beta),\cF)$}
\end{center} \end{figure}

\section{The case of a monomial curve}\label{case_monomial_curve}
 Let $A=(a_1 \; a_2 \; \cdots \; a_n )$ be an integer row matrix with $1 < a_1 < a_2 < \cdots < a_n$ and
assume without loss of generality  ${\rm gcd}(a_1 ,\ldots ,a_n )=1$.

In this Section we will compute de dimension of the germs of the
cohomology of ${\rm Irr}^{(s)}_Y (\HHAb )$ at any point in  $Y=\{x_n
=0\}\subseteq X=\C^n$ for all $\beta \in \C$ and $1\leq s \leq
\infty$.

We will consider the auxiliary matrix  $A'= (1 \; a_1 \;  \cdots \;
a_n )$ and the corresponding hypergeometric ideal $H_{A'} (\beta
)\subset A_{n+1}$ where $A_{n+1}$ is the Weyl algebra of linear
differential operators with coefficients in the polynomial ring $\C
[x_0 ,x_1 ,\ldots , x_n]$. We denote $\partial_0$ the partial
derivative with respect to $x_0$.

In this Section we denote $X'=\CC^{n+1}$ and we identify $X=\CC^n$
with the hyperplane $(x_0=0)$ in $X'$. If $\cD_{X'}$ is the sheaf of
linear differential operators with holomorphic coefficients in $X'$
then  the analytic hypergeometric system associated with
$(A',\beta)$, denoted by $\cM_{A'}(\beta)$, is by definition the
quotient of $\cD_{X'}$ by the sheaf of ideals generated by the
hypergeometric ideal $H_{A'}(\beta) \subset A_{n+1}$ (see Section
\ref{GGZ-GKZ-systems}).

One of the main results in this Section is the following
\begin{theorem}\label{restriccioninversa}
Let $A'= (1 \; a_1 \;  \cdots \; a_n )$ an integer row matrix with
$1<a_1 <\cdots <a_n $ and  ${\rm gcd}(a_1 ,\ldots ,a_n )=1$. For
each  $\beta \in \C$ there exists $\beta'\in \CC$ such that the
restriction of $\mathcal{M}_{A'} (\beta )$ to $X=\{ x_0 = 0\}
\subset X'$ is the $\cD_X$--module
$$\cM_{A'}(\beta)_{|X}:=\frac{\cD_{X'}}{\cD_{X'} H_{A'}(\beta) + x_0 \cD_{X'}} \simeq
\cM_A (\beta ')$$ where $A= (a_1 \; a_2 \; \cdots \; a_n)$.
Moreover,  for all but finitely many  $\beta$ we have $\beta '
=\beta$.
\end{theorem}

\begin{proof} Following \cite{SST}, we will use the notations defined
in Notation \ref{notations_SST} and Definition \ref{def-b-function}.
For $i=1,2,\ldots ,n$ let us consider $\delta_i \in \N$ the smallest
integer satisfying $1+\delta_i a_i \in \sum_{j\neq i} a_j \N$. Such
a $\delta_i$ exists because ${\rm gcd}(a_1 ,\ldots ,a_n )=1$.

Let us consider  $\rho_{i j}\in \N$ such that
$$1+\delta_i a_i =\sum_{j\neq i} \rho_{ij} a_j.$$ Then the operator
$Q_{i}:=\partial_0 \partial_i^{\delta_i }- \partial^{\rho_i}$
belongs to $I_{A'}$ where $\partial^{\rho_i}=\prod_{j\neq 0,i}
\partial_j^{\rho_{i j}}$. Moreover, for $\omega =(1,0,\ldots ,0)\in \NN^{n+1}$ we have  $\inww (Q_i
)= \partial_0 \partial_i^{\delta_i } \in \inw I_{A'}$ for
$i=1,\ldots ,n$.

We also have that $P_1 =\partial_0^{a_1} -\partial_1 \in I_{A'}$ and
$\inww P_1 = \partial_0^{a_1}\in \inw I_{A'}$. Then
\begin{align}
\inw I_{A'} \supseteq \langle
\partial_0^{a_1},
\partial_0 \partial_1^{\delta_1 } \ldots ,  \partial_0 \partial_n^{\delta_n },
T_1 ,\ldots , T_r \rangle \label{inwAprima}
\end{align} for any binomial generating system $\{ T_1 ,\ldots ,T_r \}
\subseteq \C [\partial_1 ,\ldots ,\partial_n ]$  of the ideal  $I_A
= I_{A'} \cap \C [\partial_1 ,\ldots ,\partial_n ]$ (notice that
$u\in L_A=\ker_\ZZ(A)\subset \ZZ^n \Longleftrightarrow (0,u) \in
L_{A'}=\ker_\ZZ(A')\subset \ZZ^{n+1}$).

Using (\ref{inwAprima}) we can prove (similarly to the proof of
Lemma \ref{inwwHAb}  for $k=1$) that for $\beta \notin \N^{\ast}$ or
$\beta \in \N^{\ast}$ big enough, we have
\begin{align}
\partial_0 \in {\operatorname{fin}}_{\omega}(H_{A'}(\beta ) )= \inw I_{A'} + \langle E' \rangle \label{partial0}
\end{align} where $E'=E+ x_0 \partial_0$ and  $E:=E(\beta)=\sum_{i=1}^n a_i
x_i\partial_i-\beta$. So there exists $R\in H_{A'}(\beta)$ such that
$\partial_0=\inww (R)$. In particular we have
$$\langle H_A (\beta ), \partial_0 \rangle \subseteq
{\operatorname{fin}}_{\omega}(H_{A'}(\beta ) ) \subseteq \inww
(H_{A'}(\beta ))$$ and the $b$--function of $H_{A'}(\beta)$ with
respect to $\omega$ is $b(\tau)=\tau$. So the restriction of
$\mathcal{M}_{A'}(\beta)$ to  $(x_0 =0)$ is a cyclic $\cD_X$-module
(see \cite[Algorithm 5.2.8]{SST}).

In order to compute $\mathcal{M}_{A'}(\beta)_{|(x_0 =0)}$ we will
follow \cite[Algorithm 5.2.8]{SST}. First of all we need to describe
the form of a Groebner basis of $H_{A'}(\beta)$ with respect to
$\omega$. Let $\{T_1 ,\ldots , T_r , R_1 ,\ldots ,R_\ell \}$ be a
Groebner basis of $I_{A'}$ with respect to $\omega$. So we have
$$I_{A'}=\langle T_1 ,\ldots , T_r , R_1 ,\ldots ,R_\ell\rangle$$ and
$$\inw I_{A'} =\langle T_1 ,\ldots , T_r , \inww R_1 ,\ldots ,\inww
R_\ell \rangle .$$

If, for some $i=0,\ldots,\ell$,  the $\omega$-order of  $\inww R_i$
is $0$, then $\inww R_i = R_i \in I_{A'}\cap \C [\partial_1 ,\ldots
,\partial_n ]= I_A $ and then $\inww R_i = R_i \in \langle T_1
,\ldots ,T_r\rangle $.

If the $\omega$-order of  $\inww R_i$ is greater than or equal to
$1$, then $\partial_0 $ divide  $\inww R_i$. Then, according
(\ref{partial0}),  for  $\beta \notin \N^{\ast}$ or $\beta \in
\N^{\ast}$ big enough, we have
\begin{align}
{\operatorname{fin}}_{\omega}(H_{A'}(\beta ) )= \langle
\partial_0 , E , T_1 ,\ldots, T_r \rangle =
\langle \partial_0 \rangle + A_{n+1} H_{A}(\beta )\subseteq \inww
(H_{A'}(\beta )) \label{finprima}
\end{align}

From \cite[Th. 3.1.3]{SST}, for all but finitely many  $\beta\in\C$,
we have
\begin{align}
\inww (H_{A'}(\beta ))= \langle
\partial_0 , E , T_1 ,\ldots, T_r \rangle = \langle \partial_0 \rangle + A_{n+1} H_{A}(\beta ). \label{inprima}
\end{align}

So, for all but finitely many  $\beta\in \CC$, the set  $$\cG = \{R,
R_1 ,\ldots , R_\ell , E' , T_1 ,\ldots, T_r \}$$ is a Groebner
basis of $H_{A'}(\beta)$ with respect to $\omega$, since first of
all $\cG$ is a generating system of $H_{A'}(\beta)$ and on the other
hand $\inww (H_{A'}(\beta))= A_{n+1} \inww ( \cG )$.


We can now follow \cite[Algorithm 5.2.8]{SST}, as in the proof of
Theorem \ref{sumadirecta}, to prove the result for all but finitely
many $\beta\in \CC$. Then, to finish the proof it is enough to apply
Proposition \ref{M-a-beta-isom} for $A'$.
\end{proof}

\begin{remark}\label{restriccion-a-X} Recall that $Y=(x_n=0)\subset X=\CC^n$ and
$Z=(x_{n-1}=0)\subset X$. Let us denote  $Y'=\{x_n =0\}\subset X'$,
$Z'=\{x_{n-1}=0\}\subset X' $. Notice that  $Y=Y'\cap X$ and
$Z=Z'\cap X$.

By using Cauchy-Kovalevskaya Theorem for Gevrey series (see
\cite[Cor. 2.2.4]{Laurent-Mebkhout2}), \cite[Proposition
4.2]{Castro-Takayama} and Theorem \ref{restriccioninversa}, we get,
for all but finitely many  $\beta \in \C$ and for all $1 \leq s\leq
\infty$, the following isomorphism
$$\mathbb{R}\cH om_{\cD_{X'}} (\mathcal{M}_{A'}(\beta ) ,\cO_{\widehat{X'|Y'}} (s))_{|X} \stackrel{\simeq}
{\rightarrow} \mathbb{R}\cH om_{\cD_X} (\HHAb
,\cO_{\widehat{X|Y}}(s)).
$$
\end{remark}

We also have the following

\begin{theorem}\label{restriction1n}
Let  $A=(a_1 \; a_2 \; \cdots \; a_n )$ be an integer row matrix
with  $1<a_1 < a_2 <\cdots <a_n$ and ${\rm gcd}(a_1,\ldots,a_n)=1$.
Then for all $\beta\in \C$ we have
\begin{enumerate}
\item[i)] $\mathcal{E}xt^0_{\cD_X} (\HHAb , \cQ_Y (s))=0 $ for
$1\leq s < a_n / a_{n-1}$. \item[ii)] $\mathcal{E}xt^0_{\cD_X}
(\HHAb , \cQ_Y (s))_{|Y\cap Z}=0 $ for $1\leq s\leq \infty$.
\item[iii)] $\dim_{\C}(\mathcal{E}xt^0_{\cD_X} (\HHAb ,\cQ_Y (s))_p )=a_{n-1}$ for
$a_n /a_{n-1} \leq s\leq \infty$ and  $ p\in Y\setminus Z$.
\item[iv)]  $\mathcal{E}xt^i_{\cD_X} (\HHAb , \cQ_Y
(s))=0 $, for  $i \geq 1$ and $1\leq s\leq \infty$.
\end{enumerate}
Here $Y=(x_n=0)\subset \CC^n$ and $Z=(x_{n-1}=0)\subset \CC^n$.
\end{theorem}

\begin{proof}
It follows from Remark \ref{restriccion-a-X}, Theorem \ref{teorext}
and Proposition \ref{M-a-beta-isom}.
\end{proof}



\begin{remark}\label{restricsolutions}
With the notations of Theorem \ref{restriction1n}, a basis of the
$\CC$--vector space $\cE xt^0_{\cD}(\cM_A(\beta),\cQ_Y(s))_p$ for
any $\frac{a_n}{a_{n-1}}\leq s\leq \infty$, $p\in Y\setminus Z$ and
$\beta \in \CC$ is given by the "substitution" (in a sense to be
precised) of $x_0 =0$ in the basis of $\cE xt^0_{\cD
'}(\cM_{A'}(\beta),\cQ_{Y'}(s))_{(0,p)}$ described in Theorem
\ref{basis_of_ext_i_Q_s}.

Remind that for $A'=(1\; a_1 \; \ldots \; a_n)$ and $\beta \in \CC$
the $\Gamma$--series described in Section
\ref{case-smooth-monomial-curve} are

$$\phi_{v^{j}} = (x')^{v^j}
\sum_{\stackrel{m_1,\ldots, m_{n-1},m_n \geq 0}{ _{\sum_{i\neq n-1}
a_i m_i \leq j+a_{n-1}m_{n-1}}}} \Gamma[v^j; u({\bf m})]
(x')^{u({\bf m})}$$ where $x'=(x_0,x_1,\ldots,x_n)$, \,
$v^j=(j,0,\ldots,0,\frac{\beta-j}{a_{n-1}}, 0)\in \CC^{n+1}$ for
$j=0,1,\ldots ,a_{n-1}-1$ and for ${\bf m} = (m_1,\ldots,m_n)\in
\ZZ^n$ we have
$$(x')^{u({\bf m})} = x_0^{-\sum_{i\neq n-1} a_i m_i+
a_{n-1}m_{n-1}}x_1^{m_1}\cdots x_{n-2}^{m_{n-2}} x_{n-1}^{-m_{n-1}}
x_n^{m_n}.
$$

For ${\bf m}=(m_1,\ldots,m_{n}) \in \NN^{n}$ such that
$j-\sum_{i\neq n-1} a_i m_i + a_{n-1} m_{n-1} \geq 0$ we have
$$\Gamma[v^j; u({\bf m})] = \frac{(\frac{\beta
-j}{a_{n-1}})_{m_{n-1}}\,j!} {m_1! \cdots m_{n-2}! m_n !
(j-\sum_{i\neq n-1} a_i m_i + a_{n-1} m_{n-1})!}.$$

Since $I_A = I_{A'} \cap \C [\partial_1 ,\ldots ,\partial_n ]$, if
$I_{A'}(f)=0$ then $I_{A}(f_{|x_{0}=0})=0$ for every formal power
series $f \in \cO_{\widehat{X ' |Y '} , p '}$ where $p'=(0,p) \in
Y'\cap \{x_0 =0\}=Y$.

Furthermore, a Laurent monomial $(x')^{w'}$ is annihilated by the
Euler operator associated with $(A',\beta)$ if and only if
$A'w'=\beta$ and after the substitution $x_0 =0$ this monomial
becomes zero or $x^w$ (in the case $w ' =(0 ,w )$) which are both
annihilated by the Euler operator associated with $(A,\beta)$, since
$A w = A' w'= \beta $ in the case $w ' =(0 ,w)$.

Hence, for $p\in Y$, every formal series solution $f \in
\cO_{\widehat{X ' |Y '} , (0,p)}$ of $\mathcal{M}_{A' }(\beta )$
becomes, after the substitution $x_0 =0$, a formal series solution
$f_{| x_0 =0} \in \cO_{\widehat{X |Y} , p}$ of $\mathcal{M}_{A
}(\beta )$. The analogous result is also true for convergent series
solutions at a point of $x_0 =0$.

After the substitution $x_0 =0$ in the series $\phi_{v^j}$ we get
$$\phi_{v^{j}|x_0 =0} =
\sum_{\stackrel{m_1,\ldots, m_{n-1},m_n \geq 0}{ _{\sum a_i m_i =
j+a_{n-1}m_{n-1}}}} \frac{(\frac{\beta -j}{a_{n-1}})_{m_{n-1}}j!
x_1^{m_1} \cdots x_{n-2}^{m_{n-2}} x_{n-1}^{\frac{\beta
-j}{a_{n-1}}-m_{n-1}} x_n^{m_n}}{m_1! \cdots m_{n-2}! m_n !}$$ for
$j=0,1,\ldots ,a_{n-1}-1$.

The summation before is taken over the set
$$\Delta_j:=\{(m_1,\ldots ,m_n ) \in \N^n : \sum_{i\not= n-1} a_i m_i = j+a_{n-1}m_{n-1} \}.$$
It is clear that  $(0,\ldots ,0)\in \Delta_0$ and for  $j\geq 1$,
$\Delta_j$ is a non empty set since ${\rm gcd}(a_1 ,\ldots ,a_n)=1$.
Moreover $\Delta_j$ is a countably infinite set for $j\geq 0$. To
this end take  some
%
${\underline{\lambda}}:=(\lambda_1 ,\ldots,\lambda_n )\in \Delta_j$.
Then $\underline{\lambda} + \mu(0,\ldots,0,a_n,a_{n-1})$ is also in
$\Delta_j$ for all $\mu \in \NN$.

The series  $\phi_{v^{j}|x_0 =0}$ is a Gevrey series of order $s=
\frac{a_n}{a_{n-1}}$ since  $\phi_{v^{j}}$ also is. We will see that
in fact the Gevrey index of $\phi_{v^{j}|x_0 =0}$ is
$\frac{a_n}{a_{n-1}}$ for $j=0,\ldots ,a_{n-1}-1$ such that
$\frac{\beta -j}{a_{n-1}}\not\in \N$. To this end let us consider
the subsum of $\phi_{v^{j}|x_0 =0}$ over the set of $(m_1,\ldots
,m_n )\in \NN^n$ of the form  ${\underline{\lambda}}^{(j)}+\N
(0,\ldots ,0,a_n ,a_{n-1})$ for some fixed
${\underline{\lambda}}^{(j)}\in \Delta_j$. Then  we get the series:

$$\frac{j! (\frac{\beta
-j}{a_{n-1}})_{\lambda_{n-1}^{(j)}} x_1^{\lambda^{(j)}_1} \cdots
x_{n-2}^{\lambda^{(j)}_{n-2}} x_{n-1}^{\frac{\beta
-j}{a_{n-1}}-{\lambda^{(j)}_{n-1}}}}{\lambda^{(j)}_1 ! \cdots
\lambda^{(j)}_{n-2}! }\sum_{m\geq 0} \frac{(\frac{\beta
-j}{a_{n-1}}-\lambda_{n-1}^{(j)})_{a_n m} x_{n-1}^{-a_n m}
x_n^{{\lambda^{(j)}_{n}}+ a_{n-1} m}}{(\lambda^{(j)}_n +a_{n-1} m )
!}$$ and it can be proven, by using d'Alembert ratio test, that its
Gevrey index equals $\frac{a_n}{a_{n-1}}$ at any point in
$Y\setminus Z$, for any $j=0,\ldots,a_{n-1}-1$ such that
$\frac{\beta -j}{a_{n-1}}\not\in \N$.

For all $j=0,\ldots ,a_{n-1}-1$ we have
$$\phi_{v^{j}|x_0 =0} \in x_{n-1}^{\frac{\beta -j}{a_{n-1}}}\C [[x_1 ,\ldots ,x_{n-2} ,x_{n-1}^{-1}, x_n
]]$$ and in particular these $a_{n-1}$ series are linearly
independent  and hence a basis of $\mathcal{E}xt^{0} (\mathcal{M}_A
(\beta ), \cO_{\widehat{X|Y}} (s))_p $ for $s\geq a_n /a_{n-1}$ and
$p\in Y \setminus Z$) if all of them are nonzero (see Theorem
\ref{ext0formal}).

If there exists $0\leq j \leq a_{n-1}-1$ such that $\phi_{v^{j}|x_0
=0}=0$ then we have that $\phi_{v^{j}}$ is a polynomial divisible by
$x_0$ and this happens if and only if $\beta \in \N \setminus \N A$.

In this last case we do not get a basis of $\mathcal{E}xt^{0}
(\mathcal{M}_A (\beta ), \cO_{\widehat{X|Y}} (s))_p $ by the
previous procedure. We will proceed as follows. Let us consider
$w'=(0,\omega)\in \N^{n +1}$ such that $\beta ' := \beta - A' w '
\in \Z_{< 0}$.  Then taking the basis $\{\phi_{v^j}\, \vert \,
j=0,\ldots, a_{n-1}-1\}$ of $\mathcal{E}xt^{0} (\mathcal{M}_{A'}
(\beta'), \cO_{\widehat{X'|Y'}} (s))_{(0,p)}$ given by Theorem
\ref{ext0formal} and after the substitution $x_0 =0$, we get a basis
of $\mathcal{E}xt^{0} (\mathcal{M}_A (\beta ' ), \cO_{\widehat{X|Y}}
(s))_p $ for $s\geq a_n /a_{n-1}$ and $p\in Y \setminus Z$.

Since $\beta , \beta ' \in \Z \setminus \N A$ then $\partial^{w}:
\mathcal{M}_{A}(\beta ') \rightarrow \mathcal{M}_{A}(\beta )$ is an
isomorphism (see \cite[Remark 3.6]{schulze-walther-GM-07} and
\cite[Lemma 6.2]{Berkesch}) and we can use this isomorphism  to
obtain a basis of $\mathcal{E}xt^{0} (\mathcal{M}_A (\beta ),
\cO_{\widehat{X|Y}} (s))_p.$


%
%

Using previous discussion and similar ideas to the ones of Section
\ref{case-smooth-monomial-curve} (we will use the notations therein)
we can prove the  following Theorem.
\end{remark}

\begin{theorem}\label{basis_ext_0_a1_an}
Let $A=(a_1 \; a_2 \; \cdots \; a_n )$ be an integer row matrix with
$0<a_1<a_2 <\cdots < a_n $, $Y=(x_n=0)\subset X$ and
$Z=(x_{n-1}=0)\subset X$. Then for all $p\in Y\setminus Z$, $\beta
\in \CC$  and $s\geq a_n /a_{n-1}$ we have:
\begin{enumerate}
\item[i)] If $\beta \notin \N$, then:
$$\mathcal{E}xt^0 (\HHAb ,\cQ_Y (s))_p = \bigoplus_{j=0}^{a_{n-1}-1}
\C \overline{(\phi_{v^{j}|x_0 =0})_p}.$$
\item[ii)] If  $\beta \in \N$, then there exists a unique
$q\in \{0,\ldots , a_{n-1} -1 \}$ such that $\frac{\beta -
q}{a_{n-1}}\in \N$
and we have:
$$\mathcal{E}xt^0 (\HHAb ,\cQ_Y (s))_p = \bigoplus_{q\neq j=0}^{a_{n-1}-1}
\CC \overline{(\phi_{v^{j}|x_0 =0})_p} \oplus
\CC\overline{(\phi_{\widetilde{v^q}\vert x_0=0})_p}.$$
\end{enumerate}
Here  $\overline {\phi}$ stands for the class modulo $\cO_{X|Y,p}$
of  $\phi \in \cO_{\widehat{X|Y} ,p}(s)$.
\end{theorem}
%
%

\begin{remark}\label{sol_generic_point_a1an} We can also compute the holomorphic solutions
of $\cM_A(\beta)$ at any point in $X\setminus Y$ for $A=(a_1 \; a_2
\; \ldots \; a_n)$ with $0< a_1 < a_2 < \ldots < a_n$ and for any
$\beta \in \CC$, where $Y=(x_n=0)\subset X=\CC^n$ (see \cite[Sec.
2.1]{fernandez-castro-dim2-2008} and Remark
\ref{sol_generic_point_1a2an}). As in the beginning of Section
\ref{case_monomial_curve} let us consider the auxiliary matrix
$A'=(1\; a_1 \; a_2 \; \ldots \; a_n)$ and the notation therein.

Let us consider the vectors $w^{j} = (j,0,\ldots ,0, \frac{\beta -
j}{a_n })\in \CC^{n+1}$, $j=0,1,\ldots ,a_n -1$.  Then the germs at
$p'=(0,p) \in X' \setminus Y' $ (with $p\in X\setminus Y$) of the
series solutions $\{ \phi_{w^{j}}:\; j=0,1,\ldots ,a_n -1\}$ is a
basis of $\cE xt^0_{\cD'}(\cM_{A'}(\beta),\cO_{X'})_{p'}$. Taking
$$\{ \phi_{w^{j}|x_0 =0}:\; j=0,1,\ldots ,a_n -1\}$$ we get a basis
of $\cE xt^0_{\cD}(\cM_A(\beta),\cO_X)_p$ for $\beta\in \CC$ such
that $\beta \notin \N \setminus \N A$ at any point $p\in X \setminus
Y$. When $\beta \in \N \setminus \N A$ we can proceed as in Remark
\ref{restricsolutions}.
\end{remark}

\section*{Conclusions}
1) In Sections \ref{case-smooth-monomial-curve} and
\ref{case_monomial_curve} we have proved that the irregularity
complex $\Irr_Y^{(s)}(\cM_A(\beta))$ is zero for $1\leq s <
a_{n}/a_{n-1}$ and concentrated in degree 0 for $a_{n}/a_{n-1}\leq
s\leq \infty$ (see Theorems \ref{teorext} and \ref{restriction1n}).
Here $A$ is a row integer matrix $(a_1\, a_2\, \cdots \, a_n)$  with
$0<a_1 < a_2 < \cdots < a_n$ and $\beta$ is a parameter in $\CC$. We
have reduced the case $a_1>1$ to the one where $a_1=1$ and then to
the two dimensional case treated in
\cite{fernandez-castro-dim2-2008}.

2) We have described a basis of $\cE
xt^0_{\cD_X}(\cM_A(\beta),\cQ_Y(s))_p$ for $p\in Y\setminus Z$ and
$a_{n}/a_{n-1}\leq s \leq \infty$ (see Theorems
\ref{basis_of_ext_i_Q_s} and \ref{basis_ext_0_a1_an}). Here
$Y=(x_n=0)\subset X=\CC^n$ and $Z=(x_{n-1}=0)\subset X$. From the
form of the basis it is  easy to see that the eigenvalues of the
corresponding monodromy, with respect to $Z$, are simply
$\exp(\frac{2\pi i (\beta - k)}{a_{n-1}})$ for
$k=0,\ldots,a_{n-1}-1$. Notice that for $\beta\in \ZZ$ one
eigenvalue (the one corresponding to the unique
$k=0,\ldots,a_{n-1}-1$ such that $\frac{\beta - k}{a_{n-1}}\in \ZZ$)
is just 1. See Remark \ref{restricsolutions} for notations.

\end{document}